\documentclass[a4paper,10pt]{article} 
 
\newcommand \A[1]{{\bf (#1)}}

\usepackage{fancybox} 
\usepackage{pifont} 
\usepackage{cancel} 

\usepackage{fancybox} 
\usepackage{amssymb} 
\usepackage[applemac]{inputenc}
\usepackage{amsfonts}
\usepackage{amsthm}
\usepackage{color}
\usepackage{mathrsfs} 
\usepackage{dsfont}

\usepackage{mathtools}                                                      
\mathtoolsset{showonlyrefs=true}      

\usepackage{color}

\newtheorem{theorem}{Theorem}[section]
\newtheorem{proposition}[theorem]{Proposition}
\newtheorem{corollary}[theorem]{Corollary}
\newtheorem{lemma}[theorem]{Lemma}
\newtheorem{definition}[theorem]{Definition}

\newtheorem{remark}[theorem]{Remark}

\def \a {{\alpha}}

\def \d {{\delta}}

\def \s {{\sigma}}

\def \R {{\mathbb {R}}}
\def \N {{\mathbb {N}}}

\def \x {{\xi}}
\def \e {{\varepsilon}}
\def \eps {{\varepsilon}}
\def \r {{\varrho}}

\def \t {{\tau}}
\def \t {{\tau}}

\def \div {{\text{\rm div}}}

\def \tilde {\widetilde}

\def\p{\partial}

\def \a {{\alpha}}

\def \d {{\delta}}

\def \s {{\sigma}}

\def\cA{{\mathcal A}}

\def\cC{{\mathcal C}}

\def\cL{{\mathcal L}}

\def\cS{{\mathcal S}}

\def\mD{{\mathbb D}}
\def\mE{{\mathbb E}}

\def\mI{{\mathbb I}}

\def\mN{{\mathbb N}}

\def\mP{{\mathbb P}}

\def\mR{{\mathbb R}}

\def\1{{\mathbf{1}}}

\def\sF{{\mathscr F}}

\def\E{\mathbb E}

\def\geq{\geqslant}
\def\leq{\leqslant}
\def\ge{\geqslant}
\def\le{\leqslant}

\def\div{\mathord{{\rm div}}}

\def\eps{\varepsilon}
\def\t{\tau}
\def\e{\mathrm{e}}

\def\a{\alpha}

\def\p{\partial}
\def\d{\delta}

\def\[{{\Big[}}
\def\]{{\Big]}}
\def\<{{\langle}}
\def\>{{\rangle}}
\def\({{\Big(}}
\def\){{\Big)}}

\def\bx{{\mathbf{x}}}
\def\tr{\mathrm {tr}}

\def\dif{{\mathord{{\rm d}}}}

\def\no{\nonumber}
\def\={&\!\!=\!\!&}
\def\bt{\begin{theorem}}
\def\et{\end{theorem}}
\def\bl{\begin{lemma}}
\def\el{\end{lemma}}
\def\br{\begin{remark}}
\def\er{\end{remark}}
\def\bx{\begin{Examples}}
\def\ex{\end{Examples}}
\def\bd{\begin{definition}}
\def\ed{\end{definition}}
\def\bp{\begin{proposition}}
\def\ep{\end{proposition}}
\def\bc{\begin{corollary}}
\def\ec{\end{corollary}}
\def\bpf{\begin{proof}}
\def\epf{\end{proof}}

\def\geq{\geqslant}
\def\leq{\leqslant}
\def\ge{\geqslant}
\def\le{\leqslant}

\usepackage
{amsmath}
\makeatletter
\newcommand{\leqnomode}{\tagsleft@true\let\veqno\@@leqno}
\newcommand{\reqnomode}{\tagsleft@false\let\veqno\@@eqno}
\makeatother

\numberwithin{equation}{section}

\title{\textbf{Density and gradient estimates for non degenerate Brownian SDEs with unbounded measurable drift}}
\author{\textbf{S. Menozzi}\footnote{Laboratoire de Mod\'elisation Math\'ematique d'Evry (LaMME), Universit\'e d'Evry Val d'Essonne, 23 Boulevard de France 91037 Evry, France and Laboratory of Stochastic Analysis, HSE,
Pokrovsky Blvd, 11, Moscow, Russian Federation. stephane.menozzi@univ-evry.fr}, \textbf{A. Pesce}\footnote{Dipartimento di Matematica, Piazza di Porta San Donato, 5 Bologna (Italy), antonello.pesce2@unibo.it}, \textbf{X. Zhang}\footnote{School of Mathematics and Statistics, Wuhan University, Wuhan, Hubei 430072, P.R. China, Email: XichengZhang@gmail.com}}

\begin{document}

\maketitle

\begin{abstract}
We consider non degenerate Brownian SDEs with H\"older continuous in space diffusion coefficient and unbounded drift with linear growth.
We derive two sided bounds for the associated density and pointwise controls of its derivatives up to order two under some additional spatial H\"older continuity assumptions on the drift. 
Importantly, the estimates reflect the transport of the initial condition by the unbounded drift through an auxiliary, possibly regularized, flow. 

\end{abstract}


{\small{\textbf{MSC \textcolor{black}{2010}:} Primary: 60H10, 35K10; Secondary: 60H30.}}

\reqnomode

\section{Introduction} \label{Intro}

\subsection{Statement of the problem}
We are interested in providing Aronson-like bounds and corresponding pointwise  estimates for the derivatives up to order two for the transition probability density  of 
the following $d$-dimensional, non-degenerate diffusion
\begin{equation}\label{SDE_0}
\dif X_t=b(t,X_t)\dif t+\s(t,X_t)\dif W_t, \quad t\ge 0,\ X_0=x,
\end{equation}
where $(W_t)_{t\ge 0}$ is a standard $d$-dimensional Brownian motion on the
classical Wiener space  $(\Omega,\sF,(\sF_t)_{t\ge 0},\mathbb P)$.
The diffusion coefficient $\sigma $ is assumed to be rough in time, and H\"older continuous in space. 
The drift $b$ is assumed to be measurable and to have linear growth in space. 
Importantly, we will always assume throughout the article that the diffusion coefficient $\sigma $ is bounded and separated from 0 (usual uniform ellipticity condition).

When both coefficients $b,\sigma$ are {\it bounded} and H\"older continuous, it is well known that there exists a unique weak 
solution to \eqref{SDE_0} which admits a density  (see e.g. \cite{stro:vara:79}, \cite{frie:75}, \cite{CHXZ}), i.e. for all $A\in {\mathcal B}(\R^d) $ (Borel $\sigma $-field of $\R^d $), 
$$
\mathbb P[X_t\in A|X_0=x]=\int_A p(0,x,t,y) \dif y.
$$ 
Furthermore,
it can be proved by the parametrix method that the transition density $p(0,x,t,y)$ 
enjoys the following two sided Gaussian estimates on a compact set in time:
\begin{align}\label{bound0}
C^{-1}g_{\lambda^{-1}}(t,x-y)&\le p(0,x,t,y)\le Cg_\lambda(t,x-y)
\end{align}
as well as the following gradient estimate
\begin{align}\label{bound1}
|\nabla^j_xp(0,x,t,y)|&\le Ct^{-\frac{j}{2}}g_\lambda(t,x-y), \quad j=1,2, 
\end{align}
where 
\begin{align}\label{GG1}
g_\lambda(t,x):=t^{-\frac{d}{2}}\exp\big(-\lambda |x|^2/t\big), \ \ \lambda \in (0,1], t>0,
\end{align}
and the constants $\lambda, C\ge 1$ only depend on the regularity of the coefficients, the non-degeneracy constants of the diffusion coefficients, the dimension $d$, and for the constant $C$,  
on the maximal time considered (see \cite{Friedman} and \cite{aron:59}, \cite{MR217444}).
Such methods have been successfully
applied to derive upper bounds up to the second order derivative for more general cases, such as operators satisfying a strong H\"ormander condition (see \cite{10.2307/3844906}), and also Kolmogorov operators with linear drift (see \cite{Polidoro2} and \cite{DiFrancescoPascucci2}). 
A different approach consists in viewing a logarithmic transformation of $p$ as the value function of a certain stochastic control problem (see \cite{MR806625}): such idea allows then to get the desired density estimates by choosing appropriate controls and eventually an upper bound for the \textit{logarithmic gradient} (see \cite{sheu:91}).
 
When the drift is unbounded and non-linear fewer results are available.  In fact, in this case 
it is no longer expected that the two sided estimates as given in \eqref{bound0} hold. Consider for instance the following Ornstein-Uhlenbeck (OU)-process
$$
\dif X_t=X_t\dif t+\dif W_t,\ \ X_0=x,
$$
which has, with the notations of \eqref{GG1}, the non-spatial homeogenous density 
$$
p_{\rm OU}(0,x,t,y)=(\pi (\e^{2t}-1))^{-d/2}g_{t/(\e^{2t}-1)}(t,\e^{t}x-y).
$$
In \cite{dela:meno:10} the authors derive two sided density bounds for a class of degenerate operators 
with {\it unbounded and Lipschitz drift}, satisfying a weak H\"ormander condition, by combining the two previous approaches: parametrix and logarithmic transform. 
Indeed, when the drift is unbounded it becomes difficult to get good controls for the iterated kernels in the \textit{blunt} parametrix expansion.
In our non-degenerate parabolic setting those bounds still hold provided the drift is \textit{globally} Lipschitz continuous in space. Then, they  read as:
\begin{equation}
C^{-1}g_{\lambda^{-1}}(t,\theta_{t}(x)-y)\le p(0,x,t,y)\le Cg_\lambda(t,\theta_{t}(x)-y), \label{Density_bounds0}
\end{equation}
where $\theta$ stands for the deterministic flow associated with the drift, i.e. 
$$
\dot \theta_{t}(x)=b(t,\theta_t(x)),\ t\ge 0, \ \theta_0(x)=x 
$$
and $C,\lambda>0$ enjoy the same type of dependence as in \eqref{bound0} but importantly $\lambda$ now also depends on the maximal time considered. 
This means that the diffusion starting from $x$, oscillates around $\theta_t(x)$ at time $t$ with fluctuations of order $t^{-\frac{1}{2}}$.
Notice that if $b$ is bounded, then \eqref{Density_bounds0}
reduces to \eqref{bound0} since
$$
t^{-\frac 12}|x-y| -\|b\|_\infty t^{\frac 12}\le t^{-\frac 12}|\theta_t(x)-y|\leq t^{-\frac 12}|x-y| +\|b\|_\infty t^{\frac 12}.
$$ 
Hence, taking or not into consideration the flow does not give much additional information. The above control also clearly emphasizes 
why $C$ might depend on some maximal time interval considered. In the case where $b $ is bounded but not necessarily \textit{smooth}, 
the above bounds remain valid for \textit{any} regularizing flow. 

Diffusion with dynamics \eqref{SDE_0} and unbounded drifts appear in many applicative fields. Let us for instance mention the work \cite{gobe:02} which was concerned with issues related  to statistics of diffusions. We can also refer to \cite{page:panl:20} for the numerical approximation of ergodic diffusions.

In such frameworks,  estimates on the density and its derivatives are naturally required. 
Some  gradient estimates of the density were established in \cite{gobe:02}. The approach developed therein relies
on the Malliavin calculus. It  thus required some extra regularity on the drift. Also, since the deterministic flow was not taken into consideration an additional  penalizing exponential term in the right hand side\footnote{r.h.s. in short} of the bounds appeared. 
Similar features appeared in the work \cite{deck:krus:02} which established 
the existence of fundamental solutions for a strictly sublinear H\"older continuous drift.

The point of the current work is to establish estimates for the derivatives that reflect both the singularities associated with the differentiation in the parabolic setting, as in equation \eqref{bound0} above, and also reflect the key importance of the flow for unbounded drifts as it appears in the two-sided heat kernel estimate \eqref{Density_bounds0}. To the best of our knowledge, our is one of the first results for the derivatives of heat kernels with unbounded drifts.

We can actually address various frameworks. We manage to obtain two-sided heat kernel bounds for a H\"older continuous in space diffusion coefficient $\sigma $ in \eqref{SDE_0} and a drift $b$ which is uniformly bounded in time at the origin and has  linear growth in space (see assumptions {\bf (H$^\sigma_{\alpha}$)} and {\bf (H$^b_\beta$)} below). Importantly, when the drift $b$ is itself not smooth, the heat kernel bounds can be stated in the form \eqref{Density_bounds0} for any flow associated with a mollification of $b$. In particular, if the drift is continuous in space they actually hold for \textit{any} Peano flow. 
Those conditions are also sufficient to obtain gradient bounds w.r.t. the \textit{backward} variable $x$. To derive controls for the second order derivatives w.r.t. $x$ an additional spatial
H\"older continuity assumption naturally appears for the drift. Eventually, imposing some additional spatial smoothness on the diffusion coefficient, 
we also succeed to establish a gradient bound w.r.t. the \textit{forward} variable $y$.  

The paper is organized as follows. 
Our main results are stated in details in Section \ref{results}; 
Section \ref{MR_LIPSCHITZ_DRIFT} is dedicated to the proof of our main results when the coefficients satisfy our previous assumptions and are also \textit{smooth}. Importantly, we prove that the two-sided heat-kernel bounds do not depend on the smoothness of the coefficients but only on constants appearing in {\bf (H$^\sigma_{\alpha}$)} and {\bf (H$^b_\beta$)}, the fixed final time horizon $T>0$ and the dimension $d$. We also establish 
there bounds for the derivatives through Malliavin calculus techniques 
which is precisely possible because the coefficients are smooth. Those bounds serve as \textit{a priori} controls to derive in Section \ref{SEC_MT_CIRC}, through a circular type argument based on the \textit{Duhamel-parametrix type} representation of the density, that those bounds actually do not depend  on 
the smoothness of the coefficients. We then deduce the main results passing to the limit in a mollification procedure through convergence in law and compactness arguments. We eventually discuss in Section \ref{EXT} some possible extensions for the estimation of  higher order derivatives
of the heat kernel when the coefficients have some additional smoothness properties.

Let us mention that our approach developed here had previously been successfully used in \cite{chau:17} and \cite{Raynal2018SharpSE} to derive respectively the strong well-posedness of kinetic degenerate SDEs or Schauder estimates for degenerate Kolmogorov equations. It actually  seems sufficiently robust to be generalized, as soon as some suitable two sided bounds hold, in order to obtain estimates of the derivatives of the density.

\subsection{Assumptions and Main Results}\label{results} 

We make the following assumptions about $\sigma$ and $b$ in \eqref{SDE_0}.
\begin{enumerate}\label{Assumption1}
\item[{\bf (H$^\sigma_{\alpha}$)}] (Non degeneracy). There exists a positive constant $\kappa_0\ge 1 $, such that 
\begin{equation}\label{a1}
\kappa_0^{-1} |\x|^2\le \langle\s\s^{\ast}(t,x)\x,\x\rangle \le \kappa_0 |\x|^2, \quad x,\x\in \R^d, \; t\geq 0,
\end{equation}
and for some $\alpha\in(0,1)$,
\begin{equation}\label{a2}
|\s(t,x)-\s(t,y)|\le \kappa_0|x-y|^\alpha,\ \ t\geq 0,\ x,y\in\R^d.
\end{equation}
\item[{\bf (H$^b_\beta$)}] 
There exist positive constant $\kappa_1>0$ and $\beta\in[0,1]$ such that for all $x,y\in\R^d$ and $t\geq 0$,
\begin{equation}\label{a3}
|b(t,0)|\leq\kappa_1,\ |b(t,x)-b(t,y)|\leq\kappa_1(|x-y|^\beta\vee|x-y|).
\end{equation}
\end{enumerate}
It should be noticed that under {\bf (H$^b_0$)}, 
$b$ can possible be an {\it unbounded} measurable function with linear growth. For instance, $b(t,x)=x+b_0(t,x)$ with $b_0$ being bounded measurable satisfies
\eqref{a3}.
The drift $b(t,x)= c_1(t)+c_2(t)|x|^\beta,\ \beta\in [0,1] $ 
where $c_1,c_2 $ are bounded measurable  functions of time, also enters this class.
We mention that the second condition in \eqref{a3} could be stated only \textit{locally}, i.e. for e.g. $|x-y|\le 1 $. This could be checked throughout the proofs below, we prefer to keep it global for simplicity.

Moreover, under {\bf (H$^\sigma_{\alpha}$)} and {\bf (H$^b_0$)}, for
any $(s,x)\in \R_+\times \R^d$, it is well known that there exists a unique weak solution to \eqref{SDE_0} starting from $x$ at time $s$ (see e.g. \cite{stro:vara:79}, \cite{bass:perk:09}, \cite{dela:meno:10}, \cite{meno:11}). \\

For any $T\in(0,\infty]$ and $\eps\in[0,T)$, we write
$$
\mD^T_\eps:=\{(s,t)\in[0,\infty)^2: \eps<t-s<T\}.
$$

To state our main result, we need to prepare some deterministic regularized
flow associated with the drift $b$. Let $\rho$ be a nonnegative smooth function with support in the unit ball of $\R^d $ and such that $\int_{\mR^d}\rho(x)\dif x=1$. For $\eps\in(0,1]$, define
\begin{equation}\label{DEF_RHO_EPS_AND_MOLL_DRIFT}
\rho_\eps(x):=\eps^{-d}\rho(\eps^{-1}x),\ \ b_\eps(t,x):=b(t,\cdot)*\rho_\eps(x)=\int_{\mR^d} b(t,y)\rho_\eps(x-y)dy,
\end{equation}
i.e. $*$ stands for the usual spatial convolution.
Then for each $n=1,2,\cdots$, it is easy to see that
\begin{align}
|\nabla^n_x b_\eps(t,x)|&=\left|\int_{\mR^d}(b(t,y)-b(t,x))\nabla_x^n\rho_\eps(x-y)\dif y\right|\no
\\&\leq\int_{\mR^d}|b(t,y)-b(t,x)||\nabla_x^n\rho_\eps|(x-y)\dif y\no
\\&\leq\kappa_1\eps^\beta\int_{\mR^d}|\nabla_x^n\rho_\eps|(x-y)\dif y\leq  c\eps^{\beta-n}.\label{DA1}
\end{align}
On the other hand, from \eqref{a3} we also have
\begin{align}\label{DA2}
|b_\eps(t,x)-b(t,x)|\leq \int_{\mR^d}|b(t,y)-b(t,x)|\rho_\eps(x-y)\dif y\leq\kappa_1\eps^\beta.
\end{align}
For fixed $(s,x)\in \R_+ \times \R^d$, we denote by $\theta^{(\eps)}_{t,s}(x)$ the deterministic flow solving
\begin{equation}\label{flow}
\dot{\theta}^{(\eps)}_{t,s}(x)=b_\eps(t,\theta^{(\eps)}_{t,s}(x)), \; t\geq 0, \quad \theta^{(\eps)}_{s,s}(x)=x.
\end{equation}
Note that $(\theta^{(\eps)}_{t,s}(x))_{t\geq s}$ stands for a forward flow and $(\theta^{(\eps)}_{t,s}(x))_{t\leq s}$ stands for a backward flow.
Also, since we have regularized equation \eqref{flow} is well posed.

The following lemma, which provides a kind of equivalence between mollified flows, is our starting point for treating the unbounded rough drifts.
\begin{lemma}[Equivalence of flows]\label{Pr1}
Under {\bf (H$^b_0$)}, for any $\eps\in(0,1]$, the mapping $x\mapsto \theta^{(\eps)}_{t,s}(x)$ is a 
$C^\infty$-diffeomorphism and its inverse is given by $x\mapsto\theta^{(\eps)}_{s,t}(x)$.
Moreover, for any $T>0$,
there exists a constant $C=C(T,\kappa_1,d)\ge 1$ 
such that for any $\eps\in(0,1]$, all $|t-s|\leq T$ and $x,y\in\R^d$,
\begin{equation}\label{EQ1}
|\theta^{(1)}_{t,s}(x)-y|+|t-s|\asymp_C |\theta^{(\eps)}_{t,s}(x)-y|+|t-s|\asymp_C |x-\theta^{(\eps)}_{s,t}(y)|+|t-s|,
\end{equation}
where $Q_1\asymp_C Q_2$ means that $C^{-1}Q_2\leq Q_1\leq C Q_2$.
\end{lemma}
\proof
By \eqref{DA1}, it is a classical fact that $x\mapsto \theta^{(\eps)}_{t,s}(x)$  is a $C^\infty$-diffemorpihsm 
and its inverse is given by $x\mapsto\theta^{(\eps)}_{s,t}(x)$.
Below, without loss of generality, we assume $s<t$. By \eqref{flow} and \eqref{DEF_RHO_EPS_AND_MOLL_DRIFT}, 
\eqref{DA1}, \eqref{DA2}, we have
\begin{align}
|\theta^{(\eps)}_{t,s}(x)-\theta^{(1)}_{t,s}(x)|
&\leq \int_s^t\left|b_\eps(r,\theta^{(\eps)}_{r,s}(x))-b_1(r,\theta^{(\eps)}_{r,s}(x)) \right|\dif r
+ \int_s^t\left|b_1(r,\theta^{(\eps)}_{r,s}(x))-b_1(r,\theta^{(1)}_{r,s}(x)) \right|\dif r\\
&\le 2\kappa_1(t-s) + \|\nabla b_1\|_\infty  \int_s^t|\theta^{(\eps)}_{r,s}(x)-\theta^{(1)}_{r,s}(x)|\dif r,
\end{align}
which implies by the Gronwall inequality that
$$
|\theta^{(\eps)}_{t,s}(x)-\theta^{(1)}_{t,s}(x)|\leq 2\kappa_1(t-s)\e^{\|\nabla b_1\|_\infty(t-s)}.
$$
Thus we obtain
$$
|\theta^{(1)}_{t,s}(x)-y|\le |\theta^{(\eps)}_{t,s}(x)-y|+2\kappa_1 \e^{\|\nabla b_1\|_\infty (t-s)}|t-s|.
$$
By symmetry, we obtain the first $\asymp_C$. For the second one, note that by the Gronwall inequality, 
\begin{align}\label{JG11}
|\theta^{(1)}_{t,s}(x)-\theta^{(1)}_{t,s}(y)|\asymp_{\e^{\|\nabla b_1\|_\infty(t-s)}}|x-y|\Rightarrow |\theta_{t,s}^{(1)}(x)-y|\asymp_{\e^{\|\nabla b_1\|_\infty(t-s)}} |x-\theta_{s,t}^{(1)}(y)|.
\end{align}
From this, by the first $\asymp_C$, we obtain the second $\asymp_C$.
\endproof

We introduce for notational convenience the following parameter set which gathers 
important quantities appearing in the assumptions:
\begin{equation}\label{DEF_THETA}
\Theta:=(T,\alpha,\beta,\kappa_0,\kappa_1,d),
\end{equation}
where again $T>0 $ stands for the fixed considered final time, $\alpha $ denotes the H\"older regularity index of the diffusion coefficient 
$\sigma$ (see \eqref{a2}), $\kappa_0 $ is the uniform ellipticity constant in \eqref{a1}, $\kappa_1 $ and $\beta $ are related to the behavior of the drift $b$ in \eqref{a3}, and $d$ is the current underlying dimension.\\

Our main result is the following theorem.
\begin{theorem}\label{th1}
Let $\alpha\in(0,1]$. Under {\bf (H$^\sigma_{\alpha}$)} and {\bf (H$^b_0$)}, 
for any $T>0$, $(s,t)\in\mD^T_0$ and $x\in\R^d$, the unique weak solution $X_{t,s}(x)$ of \eqref{SDE_0} starting from $x$ at time $s$ admits
a density $p(s,x,t,y)$ which is continuous in $x,y\in\mR^d$.
Moreover, $p(s,x,t,y)$ enjoys the following estimates:
\begin{enumerate}
\item [(i)] ({Two-sided density bounds}) There exist constants  $\lambda_0\in (0,1],\ C_0\ge 1$ depending on $\Theta$ such that 
for any $(s,t)\in\mD^T_0$ and $x,y\in \R^d$,
\begin{align}
C^{-1}_0g_{\lambda_0^{-1}}(t-s,\theta^{(1)}_{t,s}(x)-y)&\le p(s,x,t,y)
\le C_0g_{\lambda_0}(t-s,\theta^{(1)}_{t,s}(x)-y). \label{Density_bounds_THM}
\end{align}
\item [(ii)] (Gradient estimate in $x$) There exist constants $\lambda_1\in (0,1],\ C_1\ge 1$ depending on $\Theta$ such that 
for any $(s,t)\in\mD^T_0$ and $x,y\in \R^d$,
\begin{align}
\left|\nabla_xp(s,x,t,y) \right|&\le C_1(t-s)^{-\frac{1}{2}}g_{\lambda_1}(t-s,\theta^{(1)}_{t,s}(x)-y). \label{Derivatives_bounds}
\end{align}
\item[(iii)] (Second order derivative estimate in $x$) If {\bf (H$^b_\beta$)} holds for some $\beta\in(0,1]$, then 
there exist constants $\lambda_2, C_2>0$ depending on $\Theta$ such that for any $(s,t)\in\mD^T_0$ and $x,y\in \R^d$,
\begin{align}
\left|\nabla^2_xp(s,x,t,y) \right|&\le C_2(t-s)^{-1}g_{\lambda_2}(t-s,\theta^{(1)}_{t,s}(x)-y). \label{Second_bounds}
\end{align}
\item [(iv)] (Gradient estimate in $y$) If {\bf (H$^b_\beta$)} holds for some $\beta\in(0,1]$ and for some $\alpha\in(0,1)$ and $\kappa_2>0$,
\begin{equation}\label{BD_DINI_GRAD_SIGMA}
\|\nabla\sigma\|_\infty\leq\kappa_2,\ \ |\nabla\sigma(t,x)-\nabla\sigma(t,y)|\leq \kappa_2|x-y|^{\alpha}, 
\end{equation}
then there exists constants $\lambda_3\in (0,1],\ C_3\ge 1$ depending on $\Theta$ and $\kappa_2$ such that 
for any $(s,t)\in\mD^T_0$ and $x,y\in \R^d$,
\begin{align}
\left|\nabla_yp(s,x,t,y) \right|&\le C_3(t-s)^{-\frac{1}{2}}g_{\lambda_3}(t-s,\theta^{(1)}_{t,s}(x)-y). \label{Derivatives_bounds2}
\end{align}
\end{enumerate}
\end{theorem}

\br\rm
By Lemma \ref{Pr1}, the above $\theta^{(1)}_{t,s}(x)$ can be replaced by any regularized flow $\theta^{(\eps)}_{t,s}(x)$. 
Importantly, 
if $b $ satisfies {\bf (H$^b_\beta$)} for some $\beta\in(0,1]$, 
then $\theta^{(1)}_{t,s}(x)$
can be replaced as well by \textbf{any} Peano flow solving $\dot \theta_{t,s}(x)=b(t,\theta_{t,s}(x)),\ \theta_{s,s}(x) =x$. Indeed, it is plain to check 
that, in this case, the result of Lemma \ref{Pr1} still holds with $\theta_{t,s}(x)$ instead of $\theta^{(\eps)}_{t,s}(x)$.
\er

\br\rm
Under the assumptions of the theorem, in fact, we can show the H\"older continuiy of
$\nabla_x p$, $\nabla^2_x p$ and $\nabla_y p$ in the variables $x$ and $y$ (see Appendix \ref{EQ_CONT_ASCOLI}).
\er


\subsection{Notations} In the following we will denote by $\langle \cdot,\cdot\rangle$ and $|\cdot| $ the Euclidean scalar product and norm on $\R^d$. We will use the notation $\nabla ,\nabla^2$ to denote respectively the gradient and Hessian matrix of real valued smooth functions on $\R^d $. By extension, we denote by $\nabla^j $ the $j^{\rm th}$ order derivative of a smooth function.
We use the letter $C$ for constants that may depend on the parameters in 
$\Theta =(T,\alpha,\beta,\kappa_0,\kappa_1,d)$ introduced in \eqref{DEF_THETA}. We reserve the notation $c$ for constant which only depend on the quantities $\alpha,\beta,\kappa_0,\kappa_1,d$ 
but  not on $T$. Both type of constants $c,C$ might change from line to line. Other possible dependencies are explicitly indicated when needed. Eventually, we will frequently use the notation $\lesssim$. For  two quantities $Q_1$ and $Q_2$, we mean by $Q_1\lesssim Q_2 $ that there exists $C$ with the same previously described dependence such that $Q_1\le CQ_2 $.

\section{A priori heat kernel estimates for SDEs with smooth coefficients}
\label{MR_LIPSCHITZ_DRIFT} 

In this section we suppose  {\bf (H$^\sigma_{\alpha}$)} and {\bf (H$^b_\beta$)}, and consider the mollified $b_\eps$ and $\sigma_\eps$.
In particular, we have

\begin{align}\label{DA4}
\kappa^{(\eps)}_j:=\sum_{k=1,\cdots,j}\left(\|\nabla^k_x b_\eps\|_\infty+\|\nabla_x^k \sigma_\eps\|_\infty\right)<\infty,\ \ \forall j\in\mN.
\end{align}
In the following,  for simplicity of notations, we shall drop the subscripts $\eps$. In other words, we assume
\leqnomode
\begin{align}
\let\veqno\eqno\label{ASS}\tag{\textbf{S}}
\mbox{$b$ and $\sigma$ satisfy {\bf (H$^\sigma_{\alpha}$)}, {\bf (H$^b_\beta$)} and \eqref{DA4}}.
\end{align}
Under \eqref{ASS}, it is well known that for each $(s,x)\in\R_+\times\R^d$, the following SDE has a unique strong solution:
\reqnomode
\begin{align}\label{SDE_1}
\dif X_{t,s}=b(t, X_{t,s})\dif t+\sigma(t, X_{t,s})\dif W_t,\ \ X_{s,s}=x,\ \ t\geq s.
\end{align}
The following theorem is well known and more or less standard in the theory of the Malliavin calculus.
We refer to \cite{Nu}, \cite[Remarks 2.1 and 2.2]{WZ} or \cite[Theorem 5.4]{Z16} for more details. 
\bt\label{Den}
Under \eqref{ASS}, for any $j,j'\in\mN$, $p>1$ and $T>0$, there is a constant $C=C(\Theta, j,j', \kappa_{j+j'})$ such that for all $(s,t)\in\mD^T_0$, $x\in\mR^d$ and $f\in C^\infty_b(\mR^d)$,
\begin{align}\label{GRA}
|\nabla^j \mE(\nabla^{j'}f(X_{t,s}(x)))|\leq C_p(t-s)^{-(j+j')/2}(\mE|f(X_{t,s}(x)|^p)^{1/p}.
\end{align}
In particular, $X_{t,s}(x)$ has a density $p(s,x,t,y)$, which is smooth in $x,y$.
\et

\br\rm
By It\^o's formula, one sees that $p(s,x,t,y)$ satisfies the backward Kolmogorov equation
\begin{equation}\label{KOLM_DIFF}
\p_s p(s,x,t,y)+\cL_{s,x}p(s,x,t,y)=0,\ {p}(s,\cdot,t,y)\longrightarrow \d_y(\cdot)\; \text{weakly as}\;  s\uparrow t,
\end{equation}
and the forward Kolmogorov equation (Fokker-Planck equation):
\begin{equation}\label{FP_EQ_DIFF_SMOOTH_COEFF}
\p_t p(s,x,t,y)\textcolor{blue}{-}\cL^*_{t,y}p(s,x,t,y)=0,\ {p}(s,x,t,\cdot)\longrightarrow \d_x(\cdot)\; \text{weakly as}\;  t\downarrow s,
\end{equation}
where, setting $a=\sigma\sigma^*/2$,
$$
\mathcal{L}_{s,x}f(x)= {\rm Tr}\big (a(s,x)\nabla_x^2 f(x)\big)+\langle b(s,x),\nabla_x f(x)\rangle
$$
and
$$
\mathcal{L}^*_{t,y}f(y)=\p_{y_i}\p_{y_j}(a_{ij}(t,y) f(y))-\div( b(t,\cdot) f)(y).
$$
Here, we use the usual Einstein convention 
for the adjoint operator.
\er
\subsection{The Duhamel representation for $p(s,x,t,y)$}
\label{SEC_FROZEN_PLUS_DUHAMEL}
Fix now $(\t,\xi)\in \R_+\times \R^d $ as \textit{freezing parameters} to be chosen later on.
Let $ \tilde{X}^{(\t,\xi)}_{t,s}(x)$ denote the process starting at $x$ at time $s$, with dynamics
\begin{equation}\label{SDE_FROZEN}
d\tilde{X}_{t,s}^{(\t,\xi)}=b(t,\theta_{t,\t}(\xi))\dif t+\s(t,\theta_{t,\t}(\xi))\dif W_t, \; t\geq s,\ \tilde{X}_{s,s}^{(\t,\xi)}=x,
\end{equation} 
i.e. $\tilde{X}_{t,s}^{(\t,\xi)}$ denotes the process derived from \eqref{SDE_1}, 
when freezing the spatial coefficients along the flow $\theta_{\cdot,\t}(\xi)$, where $\theta_{\cdot,\t}(\xi)$ 
is the unique solution of ODE \eqref{flow} corresponding to $b$.
It is clear that, for any choice of freezing couple $(\t,\xi) $, $\tilde X^{(\t,\xi)}_{t,s}$ has a 
Gaussian  density 
\begin{align}\label{GAU}
\tilde{p}^{(\t,\xi)}(s,x,t,y)=\frac{\exp\{-\<(\cC^{(\t,\xi)}_{t,s})^{-1}(\vartheta^{(\t,\xi)}_{t,s}+x-y), 
\vartheta^{(\t,\xi)}_{t,s}+x-y\>/2\}}{\sqrt{(2\pi)^d\det (\cC^{(\t,\xi)}_{t,s})}},
\end{align}
where
\begin{equation}\label{MEAN_COV_AL}
\vartheta^{(\t,\xi)}_{t,s}:=\int_s^t b(r,\theta_{r,\t}(\xi))\dif r, \quad 
\cC^{(\t,\xi)}_{t,s}:=\int_s^t\s\s^{\ast}(r,\theta_{r,\t}(\xi))\dif r.
\end{equation}
In particular,  $\tilde{p}^{(\t,\xi)}(s,x,t,y)$ satisfies  for fixed $(t,y)\in \R_+\times \R^d $:
\begin{equation}\label{Kbis_AL}
\p_s\tilde{p}^{(\t,\xi)}(s,x,t,y)+\mathcal{L}^{(\t,\xi)}_{s,x} \tilde{p}^{(\t,\xi)}(s,x,t,y)=0, \; (s,x)\in [0,t)\times \R^d, 
\end{equation}
subjected to the final condition
\begin{equation}\label{DIRAC_FROZEN}
\tilde{p}^{(\t,\xi)}(s,\cdot,t,y)\longrightarrow \d_y(\cdot)\; \text{weakly as}\;  s\uparrow t,
\end{equation}
where 
\begin{equation}\label{frozen_gen}
\mathcal{L}^{(\t,\xi)}_{s,x}={\rm{Tr}}\big(a(s,\theta_{s,\tau}(\xi))\cdot\nabla^2_x\big)+\langle b(s,\theta_{s,\t}(\xi)),\nabla_x\rangle
\end{equation}
denotes the generator of the diffusion with frozen coefficients in \eqref{SDE_FROZEN}.

The following lemma is direct by the explicit representation \eqref{GAU}, the uniform ellipticity condition \eqref{a1} and the chain rule.
\bl[A priori controls for the frozen Gaussian density]\label{Le21}
For any $j=0,1,2,\cdots$, there exist constants $\lambda_j, C_j>0$ depending only on $j,\kappa_0, d$
such that  for all $(\t,\xi)\in \R_+\times\R^d$, $(s,t)\in\mD^\infty_0$ and $x,y\in\R^d$,
\begin{align}\label{Low}
\tilde{p}^{(\t,\xi)}(s,x,t,y)&\ge C_0g_{\lambda_0^{-1}}\big(t-s, \vartheta^{(\t,\xi)}_{t,s}+x -y\big),
\end{align} 
and
\begin{align}\label{derivative_gaussian_AL}
\left|\nabla_x^j \tilde{p}^{(\t,\xi)}(s,x,t,y)\right|=
\left|\nabla_y^j \tilde{p}^{(\t,\xi)}(s,x,t,y)\right|&\le C_j(t-s)^{-\frac{j}{2}}g_{\lambda_j}\big(t-s, \vartheta^{(\t,\xi)}_{t,s}+x -y\big).
\end{align} 
Moreover, for each $j,j'\in\mN$, there are constants $C',\lambda'$ depending on $\Theta$ and $\kappa_{j'}$ such that
\begin{align}\label{D9}
\left| \nabla^j_x\nabla_\xi^{j'}\tilde{p}^{(\t,\xi)}(s,x,t,y)\right|
&\le C'(t-s)^{-\frac{j}{2}}g_{\lambda'}\big(t-s, \vartheta^{(\t,\xi)}_{t,s}+x -y\big).
\end{align} 
\el
\begin{proof}
We focus on \eqref{D9} for which it suffices to note that for any $k\in\mN$, $T>0$,
$$
|\nabla^k_\xi\vartheta^{(\t,\xi)}_{t,s}|+|\nabla^k_\xi\cC^{(\t,\xi)}_{t,s}|\leq C_k|t-s|,\ \ (s,t)\in\mD^T_0,\ (\tau,\xi)\in[s,t]\times\R^d,
$$
where the constant $C_k$ depends on the bound of $\nabla^jb$ and $\nabla^j\sigma$,  $j=1,\cdots, k$.
\end{proof}
The starting point of our analysis is the following Duhamel type representation formula which readily follows in the current \textit{smooth coefficients} setting from \eqref{KOLM_DIFF}-\eqref{FP_EQ_DIFF_SMOOTH_COEFF} and \eqref{Kbis_AL}-\eqref{DIRAC_FROZEN}:
\begin{align}
p(s,x,t,y)&=\tilde{p}^{(\t,\xi)}(s,x,t,y)+\int^t_s\int_{\R^d}\tilde{p}^{(\t,\xi)}(s,x,r,z)\cA^{(\tau,\xi)}_{r,z}p(r,z,t,y)\dif z\dif r\label{D0}
\\&=\tilde{p}^{(\t,\xi)}(s,x,t,y)+\int^t_s\int_{\R^d}p(s,x,r,z)\cA^{(\tau,\xi)}_{r,z}{\tilde p}^{(\t,\xi)}(r,z,t,y)\dif z\dif r,\label{D11}
\end{align}
where
\begin{align}\label{ABBR_2}
\cA^{(\tau,\xi)}_{r,z}:=\cL_{r,z}-\cL^{(\tau,\xi)}_{r,z}=\tr(A^{(\tau,\xi)}_{r,z}\cdot\nabla^2_z)+B^{(\tau,\xi)}_{r,z}\cdot\nabla_z
\end{align}
and
\begin{align}\label{ABBR_1}
A^{(\tau,\xi)}_{r,z}:=a(r,z)-a(r,\theta_{r,\tau}(\xi)),\ \ B^{(\tau,\xi)}_{r,z}:=b(r,z)-b(r,\theta_{r,\tau}(\xi)).
\end{align}
If we take $(\tau,\xi)=(s,x)$ in \eqref{D0} and set $\tilde p_0(s,x,t,y):=\tilde{p}^{(s,x)}(s,x,t,y)$, 
then we obtain the {\it backward} representation
$$
p(s,x,t,y)=\tilde{p}_0(s,x,t,y)+\int^t_s\int_{\R^d}\tilde{p}_0(s,x,r,z)\cA^{(s,x)}_{r,z}p(r,z,t,y)\dif z\dif r, 
$$
and in this case
\begin{align}\label{AQ1}
\vartheta^{(s,x)}_{t,s}+x-y=\int_s^t b(r,\theta_{r,s}(x))\dif r+x-y=\theta_{t,s}(x)-y;
\end{align}
it involves the \textit{forward} deterministic flow $\theta_{t,s}(x) $ in the frozen Gaussian density.
If we now take $(\tau,\xi)=(t,y)$ in \eqref{D11} and set $\tilde p_1(s,x,t,y):=\tilde{p}^{(t,y)}(s,x,t,y)$,  
we then obtain the {\it forward} representation
\begin{align}\label{D1}
p(s,x,t,y)=\tilde{p}_1(s,x,t,y)+\int^t_s\int_{\R^d}p(s,x,r,z)\cA^{(t,y)}_{r,z}{\tilde p}_1(r,z,t,y)\dif z\dif r,
\end{align}
and in this case
\begin{align}\label{AQ2}
\vartheta^{(t,y)}_{t,s}+x-y=\int_s^t b(r,\theta_{r,t}(y))\dif r+x-y=x-\theta_{s,t}(y).
\end{align}
It involves the \textit{backward} deterministic flow $\theta_{s,t}(y) $ in the frozen Gaussian density.

\subsection{Two-sided Estimates for the heat kernel}
\label{SMOOTH_HK}

We first deal here with the two-sided estimates for the density in the current \textit{smooth coefficients} setting. Importantly, we emphasize as much as possible that all the controls obtained are actually \textit{independent} of the derivatives of the coefficients, or even of the continuity of the drift $b$, but only depend on the parameters gathered in $\Theta $ introduced in \eqref{DEF_THETA}. We first iterate in Section \ref{SEC_PARAM_TWO_SIDED_HK} the Duhamel representation \eqref{D1} (\textit{forward case}) to obtain the so-called \textit{parametrix series} expansion of the density. We then give some controls related to the smoothing effects in time of the parametrix kernel. A specific feature of the heat kernels associated with unbounded drifts is that the corresponding 
parametrix series needs to be handled with care. Indeed, it is not direct to prove that it converges and some truncation step is needed. This fact was already observed in \cite{dela:meno:10} and we use here a similar kind of argument based on slightly different techniques deriving from the stochastic control representation of some Brownian functionals, see \cite{boue:dupu:98}, \cite{zhan:09} and Section \ref{CTR_STO} below.

We can assume here without loss of generality that $T\le 1$.  
Indeed, once the two-sided estimates are established in this case, they can be easily extended to any compact time interval $[0,T] $
through Gaussian convolutions using 
the scaling properties (see Lemma \ref{Le42}).

\subsubsection{Two-sided heat kernel estimates parametrix series}\label{SEC_PARAM_TWO_SIDED_HK}
For notational convenience, we write from now on for $(s,t)\in\mD^T_0$ and $x,y\in\R^d$,
\begin{equation}\label{DEF_H}
\tilde{p}_1(s,x,t,y)=\tilde{p}^{(t,y)}(s,x,t,y),\ \ H(s,x,t,y):=\cA^{(t,y)}_{s,x}{\tilde p}_1(s,x,t,y),
\end{equation}
and
\begin{equation}\label{DE1}
p\otimes H(s,x,t,y)=\int_s^t\int_{\R^d}p(s,x,r,z)H(r,z,t,y)\dif z\dif r.
\end{equation}
Thus, from the Duhamel representation \eqref{D1}, we have 
\begin{align}
p(s,x,t,y)=\tilde{p}_1(s,x,t,y)+(p\otimes H)(s,x,t,y).\label{FR}
\end{align}
For $N\geq 2$, by iterating $N-1$-times the identity \eqref{FR}, we obtain
\begin{align}
p(s,x,t,y)=\tilde{p}_1(s,x,t,y)+\sum_{j=1}^{N-1}(\tilde p_1\otimes H^{\otimes j})(s,x,t,y)+(p\otimes H^{\otimes N})(s,x,t,y).\label{FR1}
\end{align}

We shall now use the following notational convention without mentioning the flow $\theta^{(1)}_{t,s}(x)$. 
For $(s,t)\in \mathbb D_0^T $, $x,y\in \R^d $, we define for $\lambda>0 $:
\begin{align}\label{AQ74}
{\bf g}_\lambda(s,x,t,y):=g_{\lambda}\big(t-s, \theta^{(1)}_{t,s}(x) -y\big)=(t-s)^{-\frac{d}{2}}\exp\big(-\lambda |\theta^{(1)}_{t,s}(x)-y|^2/(t-s)\big),
\end{align}
recalling \eqref{GG1} for the last equality. 
We therefore derive from 
{Lemmas \ref{Le21} and \ref{Pr1}}
the following lemma.
\bl\label{Le24}
For any $T>0$ and $j=0,1,2,\cdots$, there exist constants $\tilde \lambda_j, \tilde C_j>0$ depending only on $\Theta$
such that  for all $(s,t)\in\mD^T_0$ and $x,y\in\R^d$,
\begin{align}\label{Low_bis}
\tilde{p}_1(s,x,t,y)&\ge \tilde C_0{\bf g}_{\tilde \lambda_0^{-1}}(s,x,t,y),
\end{align} 
and for all $\alpha\in [0,1] $,
\begin{align}\label{FIRST_CTR_FROZEN}
|x-\theta_{s,t}(y)|^\alpha\left|\nabla_x^j \tilde{p}_1(s,x,t,y)\right|
&\le \tilde C_j(t-s)^{\frac{\alpha}2-\frac{j}{2}}{\bf g}_{\tilde \lambda_j}(s,x,t,y).
\end{align} 
\el
The following convolution type inequality is also an easy consequence of Lemma \ref{Pr1}.
\bl\label{CKE}
For any $T>0$, there is an $\eps=\eps(\Theta)\in(0,1)$ such that for any $\lambda>0$, there is a $C_\eps=C_\eps(\Theta,\lambda)>0$ such that 
for all $(s,t)\in\mD^T_0$, $r\in[s,t]$ and $x,y\in\mR^d$,
\begin{align}\label{CKE0}
\int_{\mR^d} {\bf g}_\lambda(s,x,r,z){\bf g}_\lambda(r,z,t,y)\dif z\leq C_\eps {\bf g}_{\eps\lambda}(s,x,t,y).
\end{align}
\el
\begin{proof}
By definition and Lemma \ref{Pr1}, we have for some $0<\eps<\eps'<1$,
\begin{align*}
\int_{\mR^d} {\bf g}_\lambda(s,x,r,z){\bf g}_\lambda(r,z,t,y)\dif z
&=\int_{\mR^d}g_{\lambda}\big(r-s, \theta^{(1)}_{r,s}(x) -z\big)g_{\lambda}\big(t-r, \theta^{(1)}_{t,r}(z) -y\big)\dif z
\\&\lesssim\int_{\mR^d}g_{\eps'\lambda}\big(r-s, \theta^{(1)}_{r,s}(x) -z\big)g_{\eps'\lambda}\big(t-r, z-\theta^{(1)}_{r,t}(y)\big)\dif z
\\&=C g_{\eps'\lambda}\big(t-s, \theta^{(1)}_{r,s}(x) -\theta^{(1)}_{r,t}(y)\big)\lesssim {\bf g}_{\eps\lambda}(s,x,t,y),
\end{align*}
where the second equality is due to the Chapman-Kolmogorov (C-K in short) property  for the Gaussian semigroup, 
and the last inequality again follows from Lemma \ref{Pr1} and the following control
\begin{align}\label{AS0}
|\theta^{(1)}_{t,s}(x)-y|=|\theta^{(1)}_{t,r}\circ\theta^{(1)}_{r,s} (x)-\theta^{(1)}_{t,r}\circ\theta^{(1)}_{r,t}(y)|\lesssim|\theta^{(1)}_{r,s}(x) -\theta^{(1)}_{r,t}(y)|.
\end{align}
The proof is complete.
\end{proof}

The next lemma is crucial since it actually allows to control the iterated convolutions of the parametrix kernel $H$ which appears  in the expansion \eqref{FR1}.
\bl[Control of the iterated parametrix kernel]\label{Le25_NEW}
Under {\bf (H$^\sigma_{\alpha}$)} and {\bf (H$^b_0$)}, for any $T>0$ and $N\in\mN$,
there are constants $C_N,\lambda_N>0$ depending only on $\Theta$ such that for all $(s,t)\in\mD^T_0$ and $x,y\in\R^d$,
$$
|H^{\otimes N}(s,x,t,y)|\leq C_N(t-s)^{-1+\frac{N\alpha}{2}} {\bf g}_{\lambda_N}(s,x,t,y),
$$
where $\lambda_N\to 0$ as $N\to\infty$.
\el
\begin{proof}
By the definition of $H$ in \eqref{DEF_H}, \eqref{ABBR_1}, Lemma \ref{Le24} and \eqref{EQ1}, 
there exists $\lambda:=\lambda (\Theta)>0$ and $C:=C(\Theta)$ s.t.
\begin{align}
|H(s,x,t,y)|
&\leq|a(s,x)-a(s,\theta_{s,t}(y))|\cdot
|\nabla^2_x{\tilde p}_1(s,x,t,y)|+|b(s,x)-b(s,\theta_{s,t}(y))|\cdot|\nabla_x{\tilde p}_1(s,x,t,y)|
\\&\lesssim|x-\theta_{s,t}(y)|^\alpha|\nabla^2_x{\tilde p}_1(s,x,t,y)|+(1+|x-\theta_{s,t}(y)|)|\nabla_x{\tilde p}_1(s,x,t,y)|\\
&\lesssim (t-s)^{-1+\frac \alpha 2}{\bf g}_{\lambda}(s,x,t,y).\label{FIRST_CTR_H}
\end{align}
This gives the stated estimate for $N=1$. 
From \eqref{FIRST_CTR_H} and by Lemma \ref{CKE},
it is readily seen that:
\begin{align}
|H^{\otimes 2}(s,x,t,y)|&
\lesssim\left(\int_s^t  (r-s)^{-1+\frac \alpha 2}(t-r)^{-1+\frac \alpha 2}\dif r\right) {\bf g}_{\eps\lambda}(s,x,t,y)
\lesssim(t-s)^{-1+\alpha} {\bf g}_{\eps\lambda}(s,x,t,y).
\end{align}
For general $N\geq 2$, by direct induction we have
\begin{align}
|H^{\otimes N}(s,x,t,y)|& \le C_N(t-s)^{-1+\frac{N\alpha}2} {\mathbf g}_{\eps^{N-1}\lambda}(s,x,t,y).
\end{align}
The proof is complete. 
\end{proof}

 From the above lemma, \eqref{FR1} and \eqref{FIRST_CTR_FROZEN}, we thus derive that for all $N\in \N $, $(s,t)\in \mathbb D_0^T,\ x,y\in \R^d $:
 \begin{align}
p(s,x,t,y)&\le \bar C {\mathbf g}_{\lambda_{N-1}}(s,x,t,y)+|p\otimes H^{\otimes N}(s,x,t,y)|,\label{PREAL_UPPER_BOUND}
\end{align}
which is \textit{almost} the expected upper-bound except that we explicitly have to control the remainder to stop the iteration at some fixed $N$ to avoid the collapse to 0 of $\lambda_N $ as $N$ goes to infinity. This is precisely the purpose of the next subsection. 

\subsubsection{Stochastic control arguments and truncation of the parametrix series}\label{CTR_STO}

In this section, we  aim at controlling the remainder term $(p\otimes H^{\otimes N})(s,x,t,y)$ in the almost Gaussian upper-bound \eqref{PREAL_UPPER_BOUND}.

To this end, we use the variational representation formula to show the a priori derivative estimates of heat kernel when the coefficients are smooth.
The idea is essentially the same as in \cite{dela:meno:10}. 
The following variational representation formula was first proved by Bou\'e and Dupuis \cite{boue:dupu:98}. The reader is referred to  \cite{zhan:09} for an extension to the abstract Wiener space.
\bt\label{Th1}
Let $F$ be a bounded Wiener functional on the classical Wiener space $(\Omega,\sF, \mP)$. Then it holds that
$$
-\ln\mE \e^{F}=\inf_{h\in\cS}\mE\left(\frac{1}{2}\int^T_0|\dot h(s)|^2\dif s-F(\omega+h)\right),
$$
where $\cS$ denotes the set of all $\R^d$-valued $\sF_t$-adapted and absolutely continuous processes with 
$$
\mE\int^T_0|\dot h(s)|^2\dif s<\infty.
$$
\et
Using the above variational representation formula, we obtain the following important lemma.
\bl\label{Le23}
Let  $\ell:\mR^d\to(0,\infty)$ be a bounded measurable function from above and below.
Under {\bf (H$^\sigma_{\alpha}$)} and {\bf (H$^b_0$)}, for any $T>0$,
there is a constant $C=C(\Theta)>0$ such that for all $x\in\mR^d$ and $(s,t)\in\mD^T_0$,
$$
\mE \ell(X_{t,s}(x))\leq C\sup_{z\in\mR^d}\exp\Big\{\ln \ell(z)-C^{-1}|z-\theta^{(1)}_{t,s}(x)|^2\Big\}.
$$
\el
\begin{proof}
Without loss of generality, we assume $s=0$ and write $X_t:=X_{t,s}(x)$. By Theorem \ref{Th1} we have
\begin{align*}
-\ln\mE \ell(X_t)=\inf_{h\in\cS}\mE\left(\frac{1}{2}\int^t_0|\dot h(s)|^2\dif s\textcolor{blue}{-}\ln \ell(X^h_t)\right),
\end{align*}
where $X^h$ solves the following SDE:
$$
\dif X^h_t=\Big(b(t,X^h_t)+\sigma(t,X^h_t)\dot h(t)\Big)\dif t+\sigma(t,X^h_t)\dif W_t,\ X^h_0=x,
$$
i.e. the control process $h$ enters the dynamics in the drift part.
Note that $\theta_t:=\theta_{t,0}(x)$ solves the following ODE:
$$
\dot \theta_t=b(t,\theta_t),\ \ \theta_0=x.
$$
By It\^o's formula, we have
\begin{align*}
\mE|X^h_t-\theta_t|^2&=\mE\int^t_0\Big(2\<X^h_s-\theta_s, b(s,X^h_s)-b(s,\theta_s)+\sigma(s,X^h_s)\dot h(s)\>+(\sigma\sigma^*)(s,X^h_s)\Big)\dif s.
\end{align*}
Recalling 
$$
|b(t,x)-b(t,y)|\leq\kappa_1(1+|x-y|),
$$
the Young inequality yields
$$
\mE|X^h_t-\theta_t|^2\lesssim\mE\int^t_0|X^h_s-\theta_s|^2\dif s+\mE\int^t_0|\dot h(s)|^2\dif s+t.
$$
From the Gronwall inequality, we thus obtain
$$
\mE|X^h_t-\theta_t|^2\lesssim\mE\int^t_0|\dot h(s)|^2\dif s+t.
$$
Hence, for some $C>0$,
$$
\frac{1}{2}\mE\int^t_0|\dot h(s)|^2\dif s\geq C^{-1}\mE|X^h_t-\theta_t|^2-C t.
$$
Therefore, we eventually derive
\begin{align*}
-\ln\mE \ell(X_t)&\geq\inf_{h\in\cS}\mE\left(C^{-1}|X^h_t-\theta_t|^2-\ln \ell(X^h_t)\right)-C
{\geq}\inf_{z\in\mR^d}\left(C^{-1}|z-\theta_t|^2-\ln \ell(z)\right)-C.
\end{align*}
The desired estimate eventually follows from Lemma \ref{Pr1}.
\end{proof}

We now state a direct yet important scaling lemma. We refer to Section 2.3 of \cite{dela:meno:10} for additional details. 
 \bl[Scaling property of the density]
\label{Le42}
 Fix $ (s,t)\in \mathbb D_0^T$ and let $\lambda:=t-s$. 
 Introduce for $u\in [0,1] $, $\widehat X_u^{\lambda}:=\lambda^{-\frac 12}X_{s+u\lambda} $. Then, $(\widehat X_u^{\lambda})_{u\in [0,1]} $ satisfies the SDE
 \begin{align}
 \dif\widehat X_u^{\lambda}&=\lambda^{\frac 12}b\big(s+u\lambda,\widehat X_u^{\lambda}\lambda^{\frac 12}\big) \dif u +\sigma\big(s+u\lambda,\widehat X_u^{\lambda}\lambda^{\frac 12}\big)\dif\widehat W_u^{\lambda}
 =\widehat b^{\lambda}(u,\widehat X_u^{\lambda}) \dif u+\widehat \sigma^{\lambda}(u,\widehat X_u^{\lambda})  \dif\widehat W_u^{\lambda},
 \end{align}
where $ \widehat W_u^{\lambda}=\lambda^{-\frac 12}W_{u\lambda}$ is a Brownian motion. It also holds that:
$$p(s,x,t,y)=\lambda^{-\frac d2} \widehat p^{\lambda}\Big(0,\lambda^{-\frac 12}x,1, \lambda^{-\frac 12}y\Big),$$
and introducing for $z\in \R^d,\ u\in [0,1],\  \partial_u { {\widehat \theta}}_{u,0}^{\lambda}(z)=\widehat b^{\lambda}(u,\widehat \theta_{u,0}(z) ),\ {\widehat \theta}_{0,0}^{\lambda}(z)=z $,
$$\big|{\widehat \theta}_{1,0}^{\lambda}(\lambda^{-\frac 12}x)-\lambda^{-\frac 12}y\big|^2=\lambda^{-1}|\theta_{t,s}(x)-y|^2 .$$
  \el
\begin{proof}We only focus on the last statement. The other ones readily follow from the change of variable. Write:
\begin{align*}
\lambda^{-\frac 12}\theta_{t,s}(x)=\lambda^{-\frac{1}{2}}x+\lambda^{-\frac{1}{2}}\int_s^t b(r,\theta_{r,s}(x)) \dif r
=\lambda^{-\frac{1}{2}}x+\lambda^{\frac 12}\int_0^1 b(s+u\lambda,\theta_{s+u\lambda,s}(x)) \dif u.
\end{align*}
Setting now for $u\in [0,1] $, $\bar \theta_{u,0}(x)=\theta_{s+u\lambda,s}(x) $, the above equation rewrites:
\begin{align*}
\lambda^{-\frac 12}\bar \theta_{1,0}(x)=\lambda^{-\frac{1}{2}}x+\lambda^{\frac 12}\int_0^1 b(s+u\lambda,\bar \theta_{u,0}(x)) \dif u
=\lambda^{-\frac{1}{2}}x+\int_0^1 \widehat b^{\lambda}(u,\lambda^{-\frac 12}\bar \theta_{u,0}(x)\big) \dif u
\end{align*}
from which we readily derive by uniqueness of the solution to the ODE that for $u\in [0,1]$,
$$
\lambda^{-\frac 12}\bar \theta_{u,0}(x)={\widehat \theta}_{u,0}^{\lambda}(\lambda^{-\frac 12}x)=\lambda^{-\frac 12}\theta_{s+u\lambda,s}(x),
$$
which gives the statement.
\end{proof}

We will now use the previous Lemmas \ref{Le23} and \ref{Le42} to establish the following result from which the Gaussian upper-bound will readily follow.

\begin{lemma}[Control of the remainder]\label{CTR_REM}
Choose $N$ large enough in order to have:
\begin{align}\label{AQ6}
-1+\frac{N\alpha}{2}>\frac{d}{2}.
\end{align}
There exists constants $C_0,\lambda_0>0$ depending only on $\Theta$ such that for all $(s,t)\in {\mathbb D}_0^T$ and $x,y\in \R^d $,
$$
|(p\otimes H^{\otimes N})(s,x,t,y)|\leq C_0{\bf g}_{\lambda_0}(s,x,t,y).
$$
\end{lemma}
\begin{proof}
From the scaling property obtained in  Lemma \ref{Le42} above, we can assume without loss of generality that $s=0$ and $t=1$.
From the definition in \eqref{DE1} and Lemmas \ref{Le25_NEW}, \ref{Le23}, we have
\begin{align*}
|(p\otimes H^{\otimes N})(0,x,1,y)|&\leq\int^1_0\left|\int_{\mR^d} p(0,x,r,z)H^{\otimes N}(r,z,1,y)\dif z\right|\dif r
=\int^1_0|\mE H^{\otimes N}(r,X_{r,0}(x),1,y)| \dif r
\\&\le C_N \int^1_0(1-r)^{-1+\frac{N\alpha}{2}}\mE  {\bf g}_{\lambda_N}(r,X_{r,0}(x),1,y) \dif r
\\&\le C_N\int^1_0(1-r)^{-1+\frac{N\alpha}{2}}\sup_{z\in\mR^d}\exp\Big\{\ln {\bf g}_{\lambda_N}(r,z,1,y)-C^{-1}|z-\theta^{(1)}_{r,0}(x)|^2\Big\} \dif r.
\end{align*}
Since by \eqref{GG1},
\begin{align*}
\ln {\bf g}_{\lambda_N}(r,z,1,y)=\ln  g_{\lambda_N}(1-r,\theta^{(1)}_{1,r}(z)-y)=-\frac{d}{2}\ln(1-r)-\lambda_N |\theta^{(1)}_{1,r}(z)-y|^2/(1-r),
\end{align*}
we have
\begin{align*}
&\sup_z(\ln g_{\lambda_N}(1-r,\theta^{(1)}_{1,r}(z)-y)-C^{-1}|z-\theta^{(1)}_{r,0}(x)|^2)
\\&\leq-\frac{d}{2}\ln(1-r)-\inf_z(\lambda_N |\theta^{(1)}_{1,r}(z)-y|^2/(1-r)+C^{-1}|z-\theta^{(1)}_{r,0}(x)|^2)
\\&\leq-\frac{d}{2}\ln(1-r)-\lambda_N'\inf_z(|z-\theta^{(1)}_{r,1}(y)|^2/(1-r)-C(1-r)+|z-\theta^{(1)}_{r,0}(x)|^2)
\\&\le -\frac{d}{2}\ln(t-r)-\lambda_N'|\theta^{(1)}_{r,1}(y)-\theta^{(1)}_{r,0}(x)|^2/2+C
\\&\leq-\frac{d}{2}\ln(t-r)-\lambda_N''|\theta^{(1)}_{1,0}(x)-y|^2+C,
\end{align*}
where the last step is due to \eqref{AS0}.
Therefore, from the condition \eqref{AQ6} and the above computations, there exist constants $C_0,\lambda_0>0$ depending only on $\Theta$ such that
$$
|(p\otimes H^{\otimes N})(0,x,1,y)|\leq C_0g_{\lambda_0}(1,\theta^{(1)}_{1,0}(x)-y)=C_0{\bf g}_{\lambda_0}(0,x,1,y).
$$
The general statement for arbitrary $(s,t)\in \mathbb D_0^T$ again follows from the scaling arguments of Lemma \ref{Le42}. 
\end{proof}

\subsubsection{Final derivation of the two-sided heat kernel estimates}
We are now in position to prove the following two-sided estimates.
\begin{theorem}\label{Th26}
Under {\bf (H$^\sigma_{\alpha}$)} and {\bf (H$^b_0$)}, for any $T>0$,
there exists constants $C_0,\lambda_0>0$ depending only on $\Theta$ such that for all $(s,t)\in\mD^T_0$ and $x,y\in\R^d$,
$$
C_0^{-1}{\bf g}_{\lambda^{-1}_0}(s,x,t,y)\le p(s,x,t,y)
\le C_0{\bf g}_{\lambda_0}(s,x,t,y). \label{Density_bounds}
$$
\end{theorem}
\begin{proof}
(i) {\bf (Upper bound)}  The upper bound is a direct consequence of the expansion \eqref{PREAL_UPPER_BOUND} and the previous Lemma \ref{CTR_REM} up to a possible modification of the constants $C_0,\lambda_0 $ that anyhow still only depend on $\Theta $.

(ii) {\bf (Lower bound)} 
By the upper bound and Lemmas \ref{Le25_NEW} and \ref{CKE}, we  get for some $\lambda_1<\lambda_0$ and $\eps\in(0,1)$:
\begin{align}
|p\otimes H(s,x,t,y)|&\lesssim \int_s^t(t-r)^{-1+\frac{\alpha}{2}}\int_{\R^d}{\bf g}_{\lambda_1}(s,x,r,z){\bf g}_{\lambda_1}(r,z,t,y)\dif z\dif r
\lesssim (t-s)^{\frac{\alpha}{2}}{\bf g}_{\eps\lambda_1}(s,x,t,y).
\end{align} 
Hence, for $|\theta_{t,s}^{(1)}(x)-y|\leq \sqrt{t-s}$, recalling \eqref{FR} and \eqref{AQ74}, we have
\begin{align*}
p(s,x,t,y)&\geq \left(C_1-C_2(t-s)^{\frac{\alpha}{2}}\right){\bf g}_{\eps\lambda_1}(s,x,t,y)
\geq \left(C_1-C_2(t-s)^{\frac{\alpha}{2}}\right)(t-s)^{-d/2}\e^{-\eps\lambda_1}.
\end{align*} 
In particular, letting $t-s\leq\delta$ with $\delta$ small enough, we obtain that 
\begin{align}\label{DP2}
p(s,x,t,y)\geq C_3(t-s)^{-d/2}\mbox{ on $\mD^\delta_0$ and $|\theta_{t,s}^{(1)}(x)-y|\leq \sqrt{t-s}$}.
\end{align}
We now precisely use a chaining argument to obtain the lower bound when $|\theta_{t,s}^{(1)}(x)-y|\geq \sqrt{t-s} $. The idea is to consider a suitable sequence of balls between the points $x$ and $y$, 
for which the diagonal lower estimate \eqref{DP2} holds,  and which also have a large enough volume to consent to derive the global \textit{off-diagonal} lower bound. The usual strategy to build such balls consists in considering the ``geodesic" line between $x$ and $y$. In the non-degenerate case, when the coefficients are bounded, this is nothing but the straight-line joining $x$ and $y$, see e.g. \cite{bass:97}. When dealing with unbounded coefficients, recall that $b$ has linear growth and is \textit{smooth}, a possibility is to consider the optimal path associated with the deterministic controllability problem $\dot \phi_u=b(u,\phi_u)+\varphi_u,\ u\in [s,t],\ \phi_s=x,\phi_t=y $ with $\varphi\in L^2([s,t],\R^d) $.  This was the choice in \cite{dela:meno:10} for a Lipschitz continuous $b$. The constants in the lower bound estimates obtained therein actually depend on the Lipschitz modulus $b$. We adopt here a slightly different strategy which only involves the \textit{mollified} flow $\theta^{(1)} $ but which will have the main advantage to provide constants that will again only depend on $\Theta$ and not on the smoothness of $b$, using thoroughly the controls established in Lemma \ref{Pr1}. We now detail such a construction which is in some sense \textit{original} though pretty natural.

From the scaling arguments of Lemma \ref{Le42}, we can assume without loss of generality that $\delta=1$, $s=0$ and $t=1$. Suppose 
$|\theta_{1,0}^{(1)}(x)-y|>1$ and let $M$ be the smallest integer greater than $4\e^{2\|\nabla b_1\|_\infty}|\theta_{1,0}^{(1)}(x)-y|^2$, i.e.,
\begin{align}\label{DP1}
M-1\leq 4\e^{2\|\nabla b_1\|_\infty}|\theta_{1,0}^{(1)}(x)-y|^2<M.
\end{align}
Importantly, we recall from \eqref{DA1} that under {\bf (H$^\sigma_{\alpha}$)} and {\bf (H$^b_0$)}, $\|\nabla b_1\|_\infty\le C(\kappa_1)$.  
Let 
$$
t_j:=j/M,\ \ j=0,1,\cdots, M.
$$
The important point for the proof is the following claim.

{\it Claim:} Set $\xi_0:=x$ and $\xi_{M}:=y$. There exist $(M+1)$-points $\xi_0,\xi_1,\cdots,\xi_{M}$ 
such that 
$$
|\xi_{j+1}-\theta_{t_{j+1},t_j}^{(1)}(\xi_j)|\leq \tfrac{1}{2\sqrt{M}},\ j=0,1,\cdots, M-1.
$$
Indeed, let $Q_1:=B_{1/(2\sqrt{M})}(\theta_{t_{1},0}^{(1)}(x))$ and recursively define for $j=2,\cdots, M$,
$$
Q_{j}:=\bigcup_{z\in Q_{j-1}} B_{1/(2\sqrt{M})}(\theta_{t_j,t_{j-1}}^{(1)}(z))=\Big\{z: {\rm dist}\Big(z, \theta_{t_j,t_{j-1}}^{(1)}(Q_{j-1})\Big)\leq 1/(2\sqrt{M})\Big\}.
$$
Letting $\kappa:=\|\nabla b_1\|_\infty$ and noting that (see \eqref{JG11})
$$
\e^{-\kappa/M}|z-z'|\leq|\theta_{t_{j+1},t_j}^{(1)}(z)-\theta_{t_{j+1},t_j}^{(1)}(z')|\leq \e^{\kappa/M}|z-z'|,
$$
by the previous induction method and noting that $\theta_{t_{j+1},t_j}^{(1)}\circ\theta_{t_j, 0}^{(1)}(x)=\theta_{t_{j+1}, 0}^{(1)}(x)$, we have
$$
B_{ j\e^{-(j-1)\kappa/M}/(2\sqrt{M})}(\theta_{t_j, 0}^{(1)}(x))\subset Q_j, \ \ j=1,2,\cdots, M.
$$

Intuitively, the image of a ball with radius $r$ under the flow $\theta_{t_j,t_{j-1}}^{(1)}$ contains a ball with radius $\e^{-\kappa/M} r$.
In particular, by \eqref{DP1},
$$
\xi_M=y\in B_{ \sqrt{M}\e^{-\kappa}/2}(\theta_{1, 0}^{(1)}(x))\subset B_{ M\e^{-(M-1)\kappa/M}/(2\sqrt{M})}(\theta_{t_M, 0}^{(1)}(x))\subset Q_M.
$$
The claim then follows. The idea of the construction is illustrated in Figure \ref{FIGURE}. 
\begin{figure}
\begin{center}
\includegraphics[scale=.6]{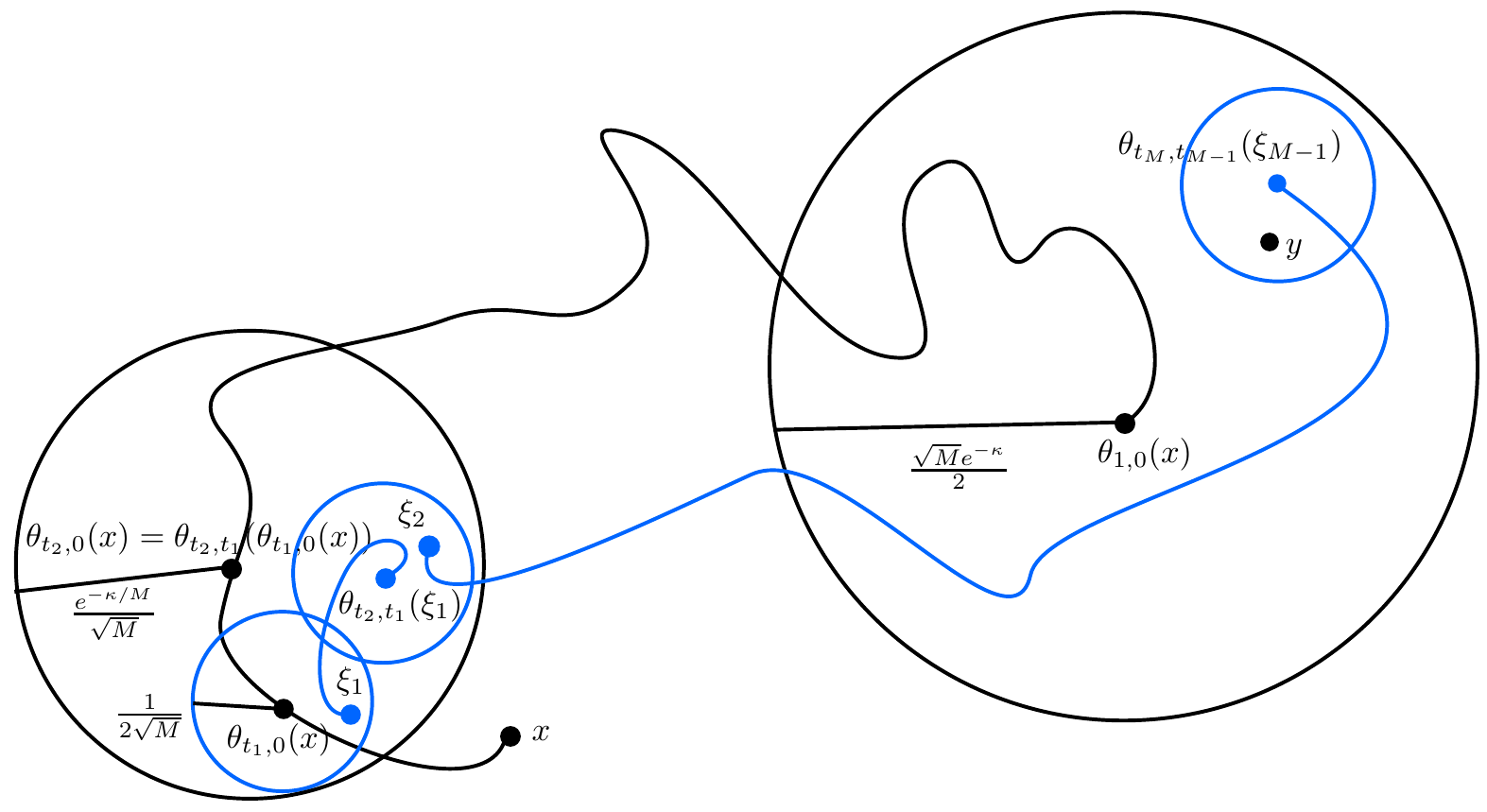}
\end{center}
\caption{Construction of the chaining balls for the lower bound.}
\label{FIGURE}
\end{figure}

Now let $\gamma:=1/(2(\e^{\|\nabla b_1\|_\infty}+1))$ and
$z_0:=x$, $z_{M+1}:=y$ and $\Sigma_j:=B_{\gamma/\sqrt{M}}(\xi_j)$. From the previous claim, we have that for $z_j\in \Sigma_j$ and $z_{j+1}\in \Sigma_{j+1}$,
\begin{align*}
|\theta_{t_{j+1}, t_j}^{(1)}(z_j)-z_{j+1}|
&\leq|\theta_{t_{j+1}, t_j}^{(1)}(z_j)-\theta_{t_{j+1}, t_j}^{(1)}(\xi_j)|+|\theta_{t_{j+1}, t_j}^{(1)}(\xi_j)-\xi_{j+1}|+|\xi_{j+1}-z_{j+1}|
\\&\leq \e^{\|\nabla b_1\|_\infty} |z_j-\xi_j|+|\theta_{t_{j+1}, t_j}^{(1)}(\xi_j)-\xi_{j+1}|+|\xi_{j+1}-z_{j+1}|
\\&\leq \frac{\gamma(\e^{\|\nabla b_1\|_\infty}+1)}{\sqrt{M}}+\frac{1}{2\sqrt{M}}=\frac{1}{\sqrt{M}}=\sqrt{t_{j+1}-t_j}.
\end{align*}
This precisely means that the previous diagonal lower bound holds for $p(t_j, z_j,t_{j+1},z_{j+1} )$. 
Thus, by the Chapman-Kolmogorov equation and \eqref{DP2}, we have
\begin{align*}
p(0,x,1,y)&=\int_{\R^d}\cdots\int_{\R^d}p(t_0,z_0,t_1,z_1)\cdots p(t_{M-1}, z_{M-1},t_{M},z_{M})\dif z_1\cdots \dif z_{M-1}
\\&\geq \int_{\Sigma_1}\cdots\int_{\Sigma_{M-1}}p(t_0,z_0,t_1,z_1)\cdots p(t_{M-1}, z_{M-1},t_{M},z_{M})\dif z_1\cdots \dif z_{M-1}
\\&\geq (C_3M^{d/2})^{M}\int_{\Sigma_1}\cdots\int_{\Sigma_{M-1}}\dif z_1\cdots \dif z_{M-1}
=(C_3M^{d/2})^{M} (M^{-d/2}\gamma^d|B_1|)^{M-1}
\\&=C^{M}_3M^{d/2} (\gamma^d|B_1|\textcolor{blue}{)}^{M-1}=M^{d/2}\exp\{M\log(C_3\gamma^d|B_1|)\}/(\gamma^d|B_1|)
\\&\geq C_4\exp\{M\log(C_3\gamma^d |B_1|)\}\geq C_5\exp\{-C_6 |\theta_{1,0}^{(1)}(x)-y|^2\},
\end{align*}
recalling the definition of $M$ in \eqref{DP1} and that $ C_3\gamma^d |B_1|\le  1$ for the last inequality.
 \end{proof}

\subsection{Estimates for the derivatives of the heat kernel with smooth coefficients}
\label{SMOOTH_DER_HK}
We insisted in the previous section on the fact that, even though we considered smooth coefficients, all our estimates for the two-sided Gaussian bounds were actually uniform w.r.t. $\Theta $ which only depend\textcolor{red}{s} on parameters appearing in {\bf (H$^b_\beta$)} and {\bf (H$^\sigma_{\alpha}$)}.

Our point of view is here different since we mainly want to derive some \textit{a priori bounds} on the derivatives of the heat-kernel  when the coefficient are smooth which will then serve in a second time, namely in the circular argument developed in Section \ref{SEC_MT_CIRC},  to prove that those estimates are actually again independent of the smoothness of the coefficients. Anyhow, in the current section, we fully exploit such a smoothness and obtain controls on the derivatives which \textit{do} depend on the derivatives of $b,\sigma $. To this end, we restart from the representation \eqref{FR} of the density and  
exploit the gradient estimate \eqref{GRA}.
 \subsubsection{Proof of the main estimates}
 \bt[Controls on the derivatives of the heat kernel with smooth coefficients]\label{CTR_DER_HK_SMOOTH_COEFF}
 Under {\bf (S)}, for $j\in \{ 1,2\}$, there exist constants $C_j:=C_j\big(\Theta,\kappa_j\big),\lambda_j:=\lambda_j (\Theta)>0,$ such that
$$
|\nabla^j_x p(s,x,t,y)|\le C_j(t-s)^{-\frac j2}{\bf g}_{\lambda_j}(s,x,t,y),\ |\nabla_y p(s,x,t,y)|\le  C_1(t-s)^{-\frac12}{\bf g}_{\lambda_1}(s,x,t,y).
$$
 \et
 \begin{proof}
 (i) Let us  first establish the estimates on the derivatives w.r.t. the backward variable $x$. 
 Write from \eqref{FR}:
 \begin{align}
 \nabla_x^j p(s,x,t,y)&=\nabla_x^j \tilde p_1(s,x,t,y)+\nabla_x^j \big(p\otimes H\big)(s,x,t,y) . 
  \label{THE_EQ_FOR_DER_SMOOTH}
 \end{align} 
 From Lemma \ref{Le24} it readily follows that
 \begin{equation}\label{MT_DER_BK}
 |\nabla_x^j \tilde p_1(s,x,t,y)|\le C_j (t-s)^{-\frac j2}{\mathbf g}_{\lambda_j}(s,x,t,y).
 \end{equation} 
 For the other contribution recall that for $u:=\frac{s+t}{2}$,
 \begin{align*}
p\otimes H(s,x,t,y)=\int_u^t  \E[H(r, X_{r,s}(x),t ,y)] \dif r+\int_s^u  \E[H(r, X_{r,s}(x),t ,y)] \dif r=:I_1(x)+I_2(x).
 \end{align*}
For $I_1$, choosing $p>1$ such that $\frac{d+\alpha-2}{2p}>\frac{d}{2}-1$, by \eqref{GRA} and Lemmas \ref{Le25_NEW} and \ref{CKE}, we get
 \begin{align}
|\nabla^j_x I_1(x)|&\lesssim \int_u^t  (r-s)^{-j/2}(\E|H(r, X_{r,s}(x),t ,y)|^p)^{1/p} \dif r\label{KM1}
\\&=\int_u^t  (r-s)^{-j/2}\left(\int_{\mR^d}p(s,x,r,z)|H(r, z,t ,y)|^p\dif z\right)^{1/p} \dif r\no
\\&\lesssim (t-s)^{-j/2}\int_u^t (t-r)^{-\frac{1}{p}+\frac{\alpha}{2p}+\frac{d}{2p}-\frac{d}{2}} \left(\int_{\mR^d}{\bf g}_{\lambda_0}(s,x,r,z) {\bf g}_{p\lambda_1}(r,z,t,y)\dif z\right)^{1/p} \dif r\no
\\&\lesssim (t-s)^{-j/2}\left(\int_u^t (t-r)^{-\frac{1}{p}+\frac{\alpha}{2p}+\frac{d}{2p}-\frac{d}{2}}  \dif r\right){\bf g}^{1/p}_{\lambda_2}(s,x,t,y)\no
\\&\lesssim (t-s)^{-j/2}(t-s)^{1-\frac{1}{p}+\frac{\alpha}{2p}+\frac{d}{2p}-\frac{d}{2}}{\bf g}^{1/p}_{\lambda_2}(s,x,t,y)\no
\lesssim (t-s)^{-j/2}{\bf g}_{\lambda_2/p}(s,x,t,y).
 \end{align}
To treat $I_2(x)$, we only consider $j=1$ since the case $j=2$ is similar. By the chain rule, we have
\begin{align*}
\nabla_x \E[ \big( H(r, X_{r,s}(x),t ,y)\big)]  =  \E[\big( \nabla_x  H\big)(r, X_{r,s}(x),t ,y)\cdot \nabla_x X_{r,s}(x) ],
\end{align*}
and for all $k\in \{1,\cdots,d\}$,
\begin{align}
\p_{x_k}H(s,x,t,y)&:=\tr(\p_{x_k}a(s,x)\cdot\nabla^2_x{\tilde p}_1(s,x,t,y)+\p_{x_k} b(s,x)\cdot\nabla_x{\tilde p}_1(s,x,t,y)
\\&\quad+\tr(a(s,x)-a(s,\theta_{s,t}(y)))\cdot\p_{x_k}\nabla^2_x{\tilde p}_1(s,x,t,y)
\\&\quad+(b(s,x)-b(s,\theta_{s,t}(y)))\cdot\p_{x_k}\nabla_x{\tilde p}_1(s,x,t,y).
\end{align}
Thus by Lemma \ref{Le24}, \eqref{DA4} and \eqref{FIRST_CTR_FROZEN}, it is easy to see that for some $\lambda_3>0$,
\begin{align}\label{GA1}
|\nabla_x H(s,x,t,y)|\lesssim (t-s)^{-1}{\bf g}_{\lambda_3}(s,x,t,y).
\end{align}
We carefully emphasize that the constant implicitly associated with the symbol $\lesssim $ here \textbf{does} depend on the smoothness of the coefficients.
Using the same argument as above, from the H\"older inequality, one sees that for $p=\frac{d}{d-1}$,
\begin{align*}
|\nabla_x I_2(x)|&\lesssim \int_s^u (\E|(\nabla_x H)(r, X_{r,s}(x),t ,y)|^p)^{1/p} \dif r
\\&\lesssim \int_s^u(t-r)^{-1} (\E{\bf g}^p_{\lambda_3}(r, X_{r,s}(x),t ,y))^{1/p} \dif r
\lesssim (t-s)^{-\frac{1}{2}}{\bf g}_{\lambda_4}(s,x,t,y).
\end{align*}
We thus obtain the gradient estimate in the variable $x$.
\\
\\
(ii) Let us now turn to the gradient estimate w.r.t. $y$. 
We restart from \eqref{D1} differentiating first w.r.t. $y$. This can be done for arbitrary freezing parameters $(\tau,\xi)$. Write:
\begin{align}
\nabla_y p(s,x,t,y)
&=\nabla_y\tilde{p}^{(\t,\xi)}(s,x,t,y)
+\int^t_s\int_{\R^d}p(s,x,r,z)\cA^{(\tau,\xi)}_{r,z}\nabla_y{\tilde p}^{(\t,\xi)}(r,z,t,y)\dif z\dif r\\
&=-\nabla_x\tilde{p}^{(\t,\xi)}(s,x,t,y)
-\int^t_s\int_{\R^d}p(s,x,r,z)\cA^{(\tau,\xi)}_{r,z}\nabla_z{\tilde p}^{(\t,\xi)}(r,z,t,y)\dif z\dif r,\label{PPX}
\end{align}
where we have used the explicit expression \eqref{GAU} for the second equality. Still letting $u=\frac{s+t}{2}$ and taking 
$(\tau,\xi)=(t,y)$, we can split
$$
\nabla_y p(s,x,t,y)=-\nabla_x\tilde{p}_1(s,x,t,y)-J_1(y)-J_2(y),
$$
where 
\begin{align*}
J_1(y)&:=\int^u_s\int_{\R^d}p(s,x,r,z)\cA^{(t,y)}_{r,z}\nabla_z{\tilde p}_1(r,z,t,y)\dif z\dif r,\\
J_2(y)&:=\int^t_u\int_{\R^d}p(s,x,r,z)\cA^{(t,y)}_{r,z}\nabla_z{\tilde p}_1(r,z,t,y)\dif z\dif r.
\end{align*}
For $J_1(y)$, from the Gaussian upper-bound of Theorem \ref{Th26},  \eqref{FIRST_CTR_FROZEN} and Lemma \ref{CKE} (see also Lemma \ref{Le25_NEW}), 
we have
\begin{align*}
|J_1(y)|\lesssim \int^u_s(t-r)^{-3/2}\int_{\R^d}{\bf g}_{\lambda_0}(s,x,r,z){\bf g}_{\lambda_5}(r,z,t,y)\dif z\dif r
\lesssim(t-s)^{-1/2}{\bf g}_{\lambda_6}(s,x,t,y).
\end{align*}
For $J_2$, integrating by parts and recalling \eqref{ABBR_2} and \eqref{ABBR_1}, we have
\begin{align*}
|J_2(y)|&\leq \int^t_u\left|\int_{\R^d}\nabla_z p(s,x,r,z)\cA^{(t,y)}_{r,z}{\tilde p}_1(r,z,t,y)\dif z\right|\dif r\\
&\quad+\int^t_u\left|\int_{\R^d} p(s,x,r,z)\nabla_zb\cdot\nabla_z{\tilde p}_1(r,z,t,y)\dif z\right|\dif r\\
&\quad+\int^t_u\left|\int_{\R^d} p(s,x,r,z)\nabla_z a\cdot\nabla^2_z{\tilde p}_1(r,z,t,y)\dif z\right|\dif r\\
&=:
J_{21}(y)+J_{22}(y)+J_{23}(y).
\end{align*}
For $J_{21}(y)$, recalling from \eqref{DEF_H} that $\cA^{(t,y)}_{r,z}{\tilde p}_1(r,z,t,y)=H(r,z,t,y)$, we derive from \eqref{GRA} and as in \eqref{KM1} that
\begin{align*}
J_{21}(y)&=\int^t_u\left|\int_{\R^d} p(s,x,r,z)\nabla_z H(r,z,t,y)\dif z\right|\dif r
=\int^t_u\left|\mE(\nabla_z H)(r,X_{r,s}(x),t,y)\right|\dif r\\
&\lesssim\int^t_u(r-s)^{-1/2}(\mE|H(r,X_{r,s}(x),t,y)|^p)^{1/p}\dif r\lesssim (t-s)^{-1/2}{\bf g}_{\lambda_7}(s,x,t,y).
\end{align*}
For $J_{22}(y)$, from the upper bound in Theorem \ref{Th26} and \eqref{FIRST_CTR_FROZEN}, we have
$$
J_{22}(y)\lesssim \int^t_u(t-r)^{-1/2}\int_{\R^d} {\bf g}_{\lambda_0}(s,x,r,z){\bf g}_{\lambda_1}(r,z,t,y)\dif z\dif r\lesssim
(t-s)^{-1/2}{\bf g}_{\lambda_8}(s,x,t,y).
$$
For $J_{23}$, since $|\nabla^2_z{\tilde p}_1(r,z,t,y)|$ has the singularity $(t-r)^{-1}$, noting that
$$
\nabla_za\cdot\nabla^2_z{\tilde p}_1=\nabla^2_z(\nabla_za\cdot{\tilde p}_1)
-\nabla^3_za\cdot{\tilde p}_1-\nabla^2_za\cdot\nabla_z{\tilde p}_1,
$$
as above, by \eqref{GRA} we still have
$$
J_{23}(y)\lesssim (t-s)^{-1/2}{\bf g}_{\lambda_9}(s,x,t,y).
$$
Combining the above estimates, we obtain the derivative estimate in $y$. The proof is complete.
 \end{proof}

\begin{remark}  
We point out that Theorem \ref{CTR_DER_HK_SMOOTH_COEFF} anyhow has some interest by itself. A careful reading of the proof shows that actually the statements about the derivatives w.r.t. $x$ hold true if additionally to {\bf (H$^\sigma_{\alpha}$)}, {\bf (H$^b_\beta$)}, the coefficients $b,\sigma$ are twice continuously differentiable with bounded derivatives and that the second order derivatives are themselves  H\"older continuous. In this framework, the Duhamel representation \eqref{FR} coupled to the heat-kernel estimates of Theorem \ref{Th26} provides an alternative approach to the \textit{full} Malliavin calculus viewpoint developed in \cite{gobe:02}. 
\end{remark}
 
 \section{Proof of Main Theorem}\label{SEC_MT_CIRC}
 
 In the following proof, the final time horizon $T>0$ is fixed.
 We first work under the assumptions {\bf (S)} aiming at obtaining constants in the estimates of Section \ref{SMOOTH_DER_HK} that only depend on $\Theta:=(T,\alpha,\beta,\kappa_0,\kappa_1,d)$ introduced in \eqref{DEF_THETA}.
 To this end, we introduce for $\delta>0$ the SDE \eqref{SDE_1} with diffusion coefficient $\sigma(t,x)=\delta\mI_{d\times d}$ 
 and denote by $\bar p_\delta$ the corresponding heat kernel.
 By the lower bound estimate proven in Theorem \ref{Th26} and scaling techniques similar to those presented in Lemma \ref{Le42}, it holds that for any $\lambda>0$, there exists $\delta:=\delta(\lambda)$ large enough and $\bar C_\delta>0, \lambda' $ depending on $\bar \Theta=(T,\beta,\delta,\kappa_1,d) $ such that
 for all $(s,t)\in \mD^T_0$ and $x,y\in\R^d$,
 \begin{equation}\label{HK_BOUNDS_DELTA}
\bar C_\delta^{-1} {\bf g}_{\lambda}(s,x,t,y)\le  \bar p_\delta(s,x,t,y)\le \bar C_\delta {\bf g}_{\lambda'}(s,x,t,y).
 \end{equation}
 We carefully mention that the Chapman-Kolmogorov equation satisfied by $\bar p_\delta$ plays a key role in the following proof when we use a Gronwall type argument. This important property had already been successfully used in \cite{pasc:pesc:19} to derive Aronson type estimates for some degenerate SPDEs. We can as well refer to \cite{pasc:pesc:20} for other applications of the parametrix method to non degenerate parabolic SPDEs.
 
 Importantly, with the notations of Section \ref{SEC_FROZEN_PLUS_DUHAMEL}, we will here choose $\lambda$, and then $\delta:=\delta(\lambda) $ s.t. for all $\gamma\in [0,1] $, $(s,t)\in \mathbb D_0^T,\ x,y\in \R^d $ and $j\in \{0,1,2\} $, 
 \begin{align}
 \label{AQ77} 
 &|\theta_{t,s}(x)-y|^\gamma|\nabla_y^j \tilde{p}_0(s,x,t,y)|+|x-\theta_{s,t}(y)|^\gamma|\nabla_x^j \tilde{p}_1(s,x,t,y)|\le 
 C_\delta (t-s)^{\frac{\gamma}2-\frac{j}{2}}\bar p_\delta(s,x,t,y),
 \end{align}
where $C_\delta$ here only depends on $\Theta = (T,\alpha,\beta,\kappa_0,\kappa_1,d)$ and $\delta,\gamma$.
 
 Without further declaration, we shall fix from now on a $\delta$ such that \eqref{AQ77} holds. From the definition of $H$ in \eqref{DEF_H} and the proof of Lemma \ref{Le25_NEW}, we also derive from this choice of $\delta $  that, under the sole assumptions {\bf (H$^\sigma_{\alpha}$)} and {\bf (H$^b_\beta$)},  there exists $C:=C(\Theta)$ s.t. for all $(s,t)\in \mathbb D_0^T,\ x,y\in \R^d $:
 \begin{equation}
\label{BD_UNIF_H}
|H(s,x,t,y)|\le C(t-s)^{-1+\frac \alpha 2}\bar p_\delta(s,x,t,y).
 \end{equation}

For simplicity we will write from now on $\bar p=\bar p_\delta$.  In particular, for all $(s,t)\in \mathbb D_0^T,\ x,y\in \R^d ,\ r\in [s,t]$:
\begin{equation}\label{CK_PROP_FOR_BAR_P}\tag{\bf{CK}}
\int_{\mR^d} \bar p(s,x,r,z)\bar p(r,z,t,y)\dif z=\bar p(s,x,t,y).
\end{equation}
For the rest of section, we use the convention that
 all the constants appearing below only depend on $\Theta=(T,\alpha,\beta,\kappa_0,\kappa_1,d)$. Again, we have shown in the previous section that for smooth coefficients the expected bounds for the derivatives hold. The constants in Theorem \ref{CTR_DER_HK_SMOOTH_COEFF} however do depend on the derivatives of the coefficients, since we use the gradient estimate \eqref{GRA}. 
 We aim here at proving that we can obtain the same type of estimates as in Theorem \ref{CTR_DER_HK_SMOOTH_COEFF} under {\bf (H$^\sigma_{\alpha}$)}, {\bf (H$^b_\beta$)} and \eqref{DA4} but for constants that only depend on $\Theta$. This is the purpose of Sections \ref{FODX} to \ref{FODY}. We will  then eventually derive in Section \ref{COMPACTNESS} the main results of Theorem \ref{th1} thanks to some compactness arguments (Ascoli-Arzel\`a theorem) thanks to the uniformity of the controls obtained for mollified parameters.
\subsection{First order derivative estimates with respect to the backward variable $x$}\label{FODX}
Without loss of generality we shall assume $s=0$ and for $t\in(0,T]$, we define
$$
f_1(t):=\sup_{x,y}|\nabla_x p(0,x,t,y)|/\bar p(0,x,t,y).
$$
From Theorem \ref{CTR_DER_HK_SMOOTH_COEFF} and \eqref{HK_BOUNDS_DELTA}, we know that
$$
\int^T_0f_1(t)\dif t<\infty.
$$
By the forward representation formula \eqref{FR}, we have
\begin{align*}
|\nabla_x p(0,x,t,y)|&\leq |\nabla_x \tilde p_1(0,x,t,y)|+|\nabla_x p|\otimes |H|(0,x,t,y).
\end{align*}
Observe first that, from Lemma \ref{Le24} and \eqref{AQ77}
$$
|\nabla_x \tilde p_1(0,x,t,y)|\lesssim t^{-1/2} {\bf g}_\lambda(0,x,t,y)\lesssim t^{-1/2}\bar p(0,x,t,y).
$$
Secondly,  \eqref{BD_UNIF_H} yields
\begin{align*}
|\nabla_x p|\otimes |H|(0,x,t,y)&\leq \int^t_0\int_{\R^d}f_1(r)\bar p(0,x,r,z) |H(r,z,t,y)|\dif z \dif r
\\&\lesssim \int^t_0 f_1(r)(t-r)^{-1+\frac{\alpha}{2}}\int_{\R^d}\bar p(0,x,r,z) \bar p(r,z,t,y)\dif z \dif r
\\&= \left(\int^t_0 f_1(r)(t-r)^{-1+\frac{\alpha}{2}}\dif r\right) \bar p(0,x,t,y),
\end{align*}
using also \eqref{CK_PROP_FOR_BAR_P} for the last identity. Thus,
\begin{align*}
f_1(t)\lesssim t^{-\frac12}+\int^t_0(t-r)^{-1+\frac{\alpha}{2}} f_1(r)\dif r.
\end{align*}
By the Volterra type Gronwall inequality, we obtain
\begin{equation}
f_1(t)\lesssim t^{-\frac12}\Rightarrow |\nabla_x p(0,x,t,y)|\lesssim t^{-\frac12}\bar p(0,x,t,y).\label{BD_GRAD_PARAM}
\end{equation}
\subsection{Second order derivative estimates with respect to the backward variable $x$}\label{SEC_FOR_BOUNDING_SECOND_ORDER_DERIVATIVES}
We assume for this section that  {\bf (H$^b_\beta$)} holds for some $\beta\in (0,1]$. It is crucial to take here $\beta>0 $.
Below we fix $t\in(0,T]$ and define for $s<t$
\begin{equation}
\label{DEF_F2}
f_2(s):=(t-s)\cdot\sup_{x,y}|\nabla^2_x p(s,x,t,y)|/\bar p(s,x,t,y).
\end{equation}
By Theorem \ref{CTR_DER_HK_SMOOTH_COEFF} and \eqref{HK_BOUNDS_DELTA}, we have
$$
\sup_{s\leq t}f_2(s)<\infty.
$$
To derive the estimate of the second order derivative of the heat kernel, we use the backward Duhamel representation \eqref{D0}. And for fixed freezing parameters  $(\tau,\xi) $ we differentiate twice w.r.t. $x$ to derive:
\begin{align}
\nabla^2_xp(s,x,t,y)
&=\nabla^2_x\tilde{p}^{(\tau,\xi)}(s,x,t,y)
+\int^t_s\int_{\R^d}\nabla^2_x\tilde{p}^{(\tau,\xi)}(s,x,r,z)\cA^{(\tau,\xi)}_{r,z}p(r,z,t,y)\dif z\dif r\no\\
&=\nabla^2_y\tilde{p}^{(\tau,\xi)}(s,x,t,y)
+\int^t_s\int_{\R^d}\nabla^2_z\tilde{p}^{(\tau,\xi)}(s,x,r,z)\cA^{(\tau,\xi)}_{r,z}p(r,z,t,y)\dif z\dif r,
\label{DECOMP_DER_2_1}
\end{align}
using again the explicit expression \eqref{GAU} for the second equality.
Let us now denote for a parameter $\eps>0$ that might depend on $r$ to be specified later on,
\begin{equation}\label{ABBR_2_BIS}
A^{\eps,(\tau,\xi)}_{r,z}:=a_\eps(r,z)-a_\eps(r,\theta_{r,\tau}(\xi)),\ \ \bar A^{(\tau,\xi)}_{r,z}:=A^{(\tau,\xi)}_{r,z}-A^{\eps,(\tau,\xi)}_{r,z},
\end{equation}
where similarly to \eqref{DEF_RHO_EPS_AND_MOLL_DRIFT},  $a_\eps(r,z)=a(r,\cdot)*\rho_\eps (z) $.
Choosing the freezing  point $(\tau,\xi)=(s,x)$ and setting as well 
$$
\tilde p_0(s,x,t,y)=\tilde{p}^{(s,x)}(s,x,t,y),\ \ u:=(t+s)/2,
$$ 
we decompose the expression in \eqref{DECOMP_DER_2_1} as follows:
\begin{align}
\nabla^2_xp(s,x,t,y)
=:\sum_{i=1}^5I_i(s,x,t,y),
\label{DECOMP_DER_2_2}
\end{align}
 where $I_1(s,x,t,y):=\nabla^2_y\tilde{p}_0(s,x,t,y)$ and
\begin{align*}
I_2(s,x,t,y)&:=\int^u_s\int_{\R^d}\nabla^2_z\tilde{p}_0(s,x,r,z)\tr(A^{(s,x)}_{r,z}\cdot\nabla^2_z p(r,z,t,y))\dif z\dif r
\\ I_3(s,x,t,y)&:=\int^t_u\int_{\R^d}\nabla^2_z\tilde{p}_0(s,x,r,z)\tr(A^{\eps,(s,x)}_{r,z}\cdot\nabla^2_z p(r,z,t,y))\dif z\dif r
\\ I_4(s,x,t,y)&:=\int^t_u\int_{\R^d}\nabla^2_z\tilde{p}_0(s,x,r,z)\tr(\bar A^{\eps,(s,x)}_{r,z}\cdot\nabla^2_z p(r,z,t,y))\dif z\dif r
\\ I_5(s,x,t,y)&:=\int^t_s\int_{\R^d}\nabla^2_z\tilde{p}_0(s,x,r,z)B^{(s,x)}_{r,z}\cdot\nabla_z p(r,z,t,y)\dif z\dif r.
\end{align*}
By Lemma \ref{Le21}, \eqref{AQ1} and \eqref{AQ77}, it is easy to see that
$$
|I_1(s,x,t,y)|\lesssim  (t-s)^{-1}g_\lambda( t-s,\theta_{t,s}(x)-y)\lesssim (t-s)^{-1}\bar p(s,x,t,y).
$$
For $I_2$, by \eqref{a2} and again \eqref{AQ77}, we have
\begin{align*}
|I_2(s,x,t,y)|&\lesssim \int^u_s\int_{\R^d}\frac{g_\lambda( r-s,\theta_{ r,s}(x)-z)}{ r-s}|z-\theta_{ r,s}(x)|^\alpha|\nabla^2_x p|(r,z,t,y))\dif z\dif r
\\&\lesssim \int^u_s \frac{(r-s)^{\alpha/2} f_2(r)}{(r-s)(t-r)}
 \int_{\R^d}\bar p(s,x,r,z)\bar p(r,z,t,y)\dif z\dif r
\\&\leq  (t-s)^{-1}\bar p(s,x,t,y)\int^t_s(r-s)^{-1+\frac{\alpha}{2}}f_2(r)\dif r.
\end{align*}
For $I_3$,  integrating by parts, we have
\begin{align*}
|I_3(s,x,t,y)|&\lesssim \int^t_u\int_{\R^d}\Big|\nabla^3_z\tilde{p}_0(s,x,r,z)
\Big|\cdot |A^{\eps,(s,x)}_{r,z}|\cdot|\nabla_z p(r,z,t,y)|\dif z\dif r
\\&+\int^t_u\int_{\R^d}\Big|\nabla^2_z\tilde{p}_0(s,x,r,z)\Big|\cdot |\nabla_zA^{\eps,(s,x)}_{r,z}|\cdot|\nabla_z p(r,z,t,y)|\dif z\dif r.
\end{align*}
Note that by the property of convolutions,
$$
|\nabla_z A^{\eps,(s,x)}_{r,z}|\lesssim \eps^{-1+\alpha},\ \ 
|A^{\eps,(s,x)}_{r,z}|\lesssim |z-\theta_{r,s}(x)|^\alpha,\ \ |\bar A^{\eps,(s,x)}_{r,z}|\lesssim\eps^\alpha.
$$
In particular, taking $\eps=(t-r)^{1/2}$, by Lemma \ref{Le21}, \eqref{AQ1}, \eqref{AQ77} and using as well the bound \eqref{BD_GRAD_PARAM} on the gradient established in the previous section, we obtain
\begin{align*}
|I_3(s,x,t,y)|&\lesssim \int^t_u\int_{\R^d}\frac{\bar p(s,x,r,z)}{(r-s)^{\frac 32}}\cdot (r-s)^{\frac \alpha2}
\cdot\frac{\bar p(r,z,t,y)}{(t-r)^{\frac 12}}\dif z\dif r
\\&\quad+\int^t_u\int_{\R^d}\frac{\bar p(s,x,r,z)}{ r-s}\cdot (t-r)^{\frac \alpha2}\cdot\frac{\bar p(r,z,t,y)}{t-r}\dif z\dif r
\\&\lesssim\bar p(s,x,t,y) \int^t_u(r-s)^{-\frac{3-\alpha}{2}}(t-r)^{-\frac{1}{2}}\dif r
\\&\quad+\bar p(s,x,t,y) \int^t_u(r-s)^{-1}(t-r)^{-1+\frac{\alpha}{2}}\dif r
\\&\lesssim\bar p(s,x,t,y)  (t-s)^{-1+\frac\alpha 2},
\end{align*}
and
\begin{align*}
|I_4(s,x,t,y)|&\lesssim \int^t_u\int_{\R^d}\frac{\bar p(s,x,r,z)}{ r-s}\cdot (t-r)^{\frac \alpha2}
\cdot\frac{f_2(r)\bar p(r,z,t,y)}{t-r}\dif z\dif r
\\&\lesssim\bar p(s,x,t,y) (t-s)^{-1}\int^t_u(t-r)^{-1+\frac{\alpha}{2}}f_2(r)\dif r.
\end{align*}
For $I_5$, from \eqref{AQ77}, 
we derive similarly to $I_2$ that
\begin{align*}
|I_5(s,x,t,y)|&\lesssim \int^t_s\int_{\R^d}\frac{g_\lambda( r-s,\theta_{ r,s}(x)-z)}{ r-s} (|z-\theta_{ r,s}(x)|^\beta+|z-\theta_{ r,s}(x)|) 
\frac{\bar p(r,z,t,y)}{(t-r)^{\frac12}}\dif z\dif r
\\&\lesssim\bar p(s,x,t,y)\int^t_s\frac{(r-s)^{\frac\beta2}+(r-s)^{\frac12}}{(r-s)(t-r)^{\frac12}}\dif r\lesssim\bar p(s,x,t,y)  (t-s)^{-1}.
\end{align*}
Combining the above estimates for the $(I_j)_{j\in \{1,\cdots,5\}} $, we obtain from \eqref{DECOMP_DER_2_2} and \eqref{DEF_F2} that:
\begin{align*}
f_2(s)\lesssim 1+\int^t_s(r-s)^{-1+\frac{\alpha}{2}}f_2(r)\dif r+\int^t_s(t-r)^{-1+\frac{\alpha}{2}}f_2(r)\dif r.
\end{align*}
Finally, from the Volterra type Gronwall inequality, we obtain 
\begin{equation}\label{BD_HESS_PARAM}
\sup_{s\in[0,t]}f_2(s)\lesssim 1\Rightarrow |\nabla^2_x p(s,x,t,y)|\lesssim (t-s)^{-1}\bar p(s,x,t,y).
\end{equation}

 \subsection{First order derivative estimate in $y$}\label{FODY}
We assume for this section that {\bf (H$^b_\beta$)} holds for some $\beta>0$ 
and the gradient of the diffusion coefficient $\sigma $ satisfies \eqref{BD_DINI_GRAD_SIGMA}.
Fix $s>0$. For $t\in(s,T]$, we define
\begin{equation}\label{DEF_F3}
f_3(t):=\sup_{x,y}|\nabla_y p(s,x,t,y)|/\bar p(s,x,t,y).
\end{equation}
By Theorem \ref{CTR_DER_HK_SMOOTH_COEFF} and \eqref{HK_BOUNDS_DELTA} we know that
$$
\int^T_0f_3(t)\dif t<\infty.
$$
In \eqref{PPX}, taking $(\tau,\xi)=(t,y)$ and recalling the notations of \eqref{ABBR_1} and 
$\tilde p_1(s,x,t,y)=\tilde p^{(t,y)}(s,x,t,y)$, by the integration by parts, we have
\begin{align}
\nabla_y p(s,x,t,y)&=-\nabla_x\tilde{p}_1(s,x,t,y)
+\int^t_s\int_{\R^d}\nabla_z p(s,x,r,z) \tr(A^{(t,y)}_{r,z}\cdot\nabla^2_z{\tilde p}_1)(r,z,t,y)\dif z\dif r
\\&\quad+\int^t_s\int_{\R^d} p(s,x,r,z) \tr((\nabla_za)(r,z)\cdot\nabla^2_z{\tilde p}_1)(r,z,t,y)\dif z\dif r
\\&\quad-\int^t_s\int_{\R^d}p(s,x,r,z)B^{(t,y)}_{r,z}\cdot\nabla^2_z{\tilde p}_1(r,z,t,y)\dif z\dif r
=:\sum_{i=1}^4J_i(s,x,t,y).\label{THE_DECOMP_GRAD_Y}
\end{align}
For $J_1$, we readily get from  \eqref{AQ77}
$$
|J_1(s,x,t,y)|\lesssim (t-s)^{-\frac12}\bar p\big(s,x,t,y\big).
$$
For $J_2$, using again \eqref{AQ77} and \eqref{DEF_F3} gives:
\begin{align*}
|J_2(s,x,t,y)|&\lesssim \int^t_s\int_{\R^d}|\nabla_z p(s,x,r,z)|\cdot(t-r)^{-1+\frac{\alpha}{2}}\bar p(r,z,t,y)\dif z\dif r
\\&\lesssim \int^t_s f_3(r)\int_{\R^d}\bar p(s,x,r,z)\cdot(t-r)^{-1+\frac{\alpha}{2}}\bar p(r,z,t,y)\dif z\dif r
\\&\lesssim \bar p(s,x,t,y)\int^t_s f_3(r)(t-r)^{-1+\frac{\alpha}{2}}\dif r.
\end{align*}
For $J_3$, we further write
\begin{align*}
J_3(s,x,t,y)
&=\int^t_s\int_{\R^d} p(s,x,r,z) \tr\Big(((\nabla_z a)(r,z)-(\nabla_z a)(r,\theta_{r,t}(y)))\cdot\nabla^2_z{\tilde p}_1\Big)(r,z,t,y)\dif z\dif r
\\&\quad+\int^t_s\int_{\R^d} p(s,x,r,z) \tr\Big((\nabla_z a)(r,\theta_{r,t}(y))\cdot\nabla^2_z{\tilde p}_1\Big)(r,z,t,y)\dif z\dif r
\\&=:J_{31}(s,x,t,y)+J_{32}(s,x,t,y).
\end{align*}
For $J_{31}$, as above, by \eqref{AQ77} we have
\begin{align*}
|J_{31}(s,x,t,y)|&\lesssim \int^t_s\int_{\R^d}\bar p(s,x,r,z)\cdot(t-r)^{\frac{\alpha}2-1}\bar p(r,z,t,y)\dif z\dif r
\lesssim \bar p(s,x,t,y).
\end{align*}
For $J_{32}$, by the integration by parts again, we derive
\begin{align*}
|J_{32}(s,x,t,y)|&\lesssim \int^t_s\int_{\R^d}|\nabla_zp(s,x,r,z)|\cdot |\nabla_z{\tilde p}_1(r,z,t,y)|\dif z\dif r
\\&\lesssim \int^t_sf_3(r)\int_{\R^d}\bar p(s,x,r,z)\cdot(t-r)^{-\frac 12}\bar p(r,z,t,y)\dif z\dif r
\\&\lesssim \bar p(s,x,t,y)\int^t_sf_3(r)(t-r)^{-\frac 12}\dif r.
\end{align*}
For $J_4$, we derive similarly to the term  $J_{31}$ that
$$
|J_{4}(s,x,t,y)|\lesssim \bar p(s,x,t,y).
$$
Combining the above 
above estimates for the $(J_i)_{i\in \{1,\cdots,4\}} $, we obtain from \eqref{THE_DECOMP_GRAD_Y} and \eqref{DEF_F3} that
$$
f_3(t)\lesssim (t-s)^{-\frac12}+\int^t_s f_3(r)(t-r)^{-1+\frac{\alpha}{2}}\dif r,
$$
which in turn yields 
\begin{equation}
f_3(t)\lesssim (t-s)^{-\frac12}\Rightarrow |\nabla_y p(s,x,t,y)|\lesssim (t-s)^{-\frac12}\bar p(s,x,t,y).\label{BD_GRAD_Y_PARAM}
\end{equation}
\subsection{Proof of Theorem \ref{th1}}\label{COMPACTNESS}
Now we turn to the notations of the beginning of Section 2 and keep the index $\eps$, associated with the spatial mollification of the coefficients. Thus,
let $p_\eps$ be the corresponding heat kernel and $X^\eps_{t,s}(x)$ the solution of
SDE \eqref{SDE_1} corresponding to $b_\eps$ and $\sigma_\eps$. It is well known, see e.g. Theorem 11.1.4 in \cite{stro:vara:79}, that under \textbf{(H$^b_\beta$)} and \textbf{(H$^\sigma_{\alpha}$)}, for any $f\in C^\infty_c(\mR^d)$
$$
\lim_{\eps\to 0}\mE f(X^\eps_{t,s}(x))=\mE f(X_{t,s}(x)).
$$
Moreover, from Theorem \ref{Th26} we have the following uniform estimate: there exist constants $\lambda_0,C_0>0$ depending only on $\Theta $ such that for all $\eps\in(0,1)$, 
$$
C^{-1}_0{\bf g}_{\lambda^{-1}_0}(s,x,t,y)\leq
p_\eps(s,x,t,y)\leq C_0{\bf g}_{\lambda_0}(s,x,t,y).
$$
Similarly, we derive from \eqref{BD_GRAD_PARAM}, \eqref{BD_HESS_PARAM} and \eqref{BD_GRAD_Y_PARAM} that 
under \textbf{(H$^\sigma_{\alpha}$)} and \textbf{(H$^b_0$)},
\begin{align}\label{Gr0}
\sup_\eps |\nabla_x p_\eps(s,x,t,y)|\leq C_1 (t-s)^{-1/2}{\bf g}_{\lambda_1}(s,x,t,y),
\end{align}
and under \textbf{(H$^\sigma_{\alpha}$)} and \textbf{(H$^b_\beta$)} with $\beta\in(0,1)$, $j\in \{1,2\} $,
\begin{align}\label{Gr}
\sup_\eps |\nabla^j_x p_\eps(s,x,t,y)|\leq C_2 (t-s)^{-j/2}{\bf g}_{\lambda_2}(s,x,t,y),
\end{align}
and under \textbf{(H$^\sigma_{\alpha}$)}, \textbf{(H$^b_\beta$)} with $\beta\in(0,1)$ and \eqref{BD_DINI_GRAD_SIGMA},
\begin{align}\label{Gr_Y}
\sup_\eps |\nabla_y p_\eps(s,x,t,y)|\leq C'_1 (t-s)^{-1/2}{\bf g}_{\lambda'_1}(s,x,t,y),
\end{align}
where in the above 
equations
\eqref{Gr0}-\eqref{Gr_Y} the constants $C_1,C_2,C'_1 $ only depend on $\Theta $ and \textit{not} on the mollification parameter $\varepsilon$. 

In particular, for nonnegative measurable function $f$, we eventually derive 
\begin{align*}
C^{-1}_0\int_{\R^d}{\bf g}_{\lambda^{-1}_0}(s,x,t,y)f(y)\dif y\leq \mE f(X_{t,s}(x))\leq C_0\int_{\R^d}{\bf g}_{\lambda_0}(s,x,t,y)f(y)\dif y,
\end{align*}
which implies that $X_{t,s}(x)$ has a density $p(s,x,t,y)$ having lower and upper bound as in \eqref{Density_bounds_THM}. This proves point \textit{(i)} of the theorem.

Moreover, for each $s<t$, we now aim at proving that 
\begin{equation}\label{PREAL_ASCOLI1}\tag{\textbf{C}$_1$}
(x,y)\mapsto\nabla_x p_\eps(s,x,t,y) \text{ is equi-continuous 
on any compact subset of}\ \mR^d\times\mR^d,
\end{equation}
and
\begin{align}
(x,y)\mapsto\nabla^2_x p_\eps(s,x,t,y) \text{ is equi-continuous on any compact subset of}\ \mR^d\times \mR^d,\label{PREAL_ASCOLI2}\tag{\textbf{C}$_2$}\\
(x,y)\mapsto\nabla_y p_\eps(s,x,t,y) \text{ is equi-continuous on any compact subset of}\ \mR^d \times \mR^d.\label{PREAL_ASCOLI3}\tag{\textbf{C}$_3$}
\end{align}
Assume for a while that such a continuity condition holds.
Then, from the Ascoli-Arzel\`a theorem, one can find a subsequence $\eps_k$ such that for each $x,y\in\mR^d$,
$$
\nabla^j_x p_{\eps_k}(s,x,t,y)\to \nabla^j_x p(s,x,t,y),\ \ j=0,1,2,\ \nabla_y p_{\eps_k}(s,x,t,y)\to \nabla_y p(s,x,t,y).
$$
The gradient and second order derivative estimates follow, under the previously recalled additional assumptions when needed, 
from \eqref{Gr0}, \eqref{Gr} and \eqref{Gr_Y}. This completes the proof of points \textit{(ii)} to \textit{(iv)} of  the theorem up to the proof of 
\eqref{PREAL_ASCOLI1}, \eqref{PREAL_ASCOLI2} and \eqref{PREAL_ASCOLI3}. 
This equicontinuity property is proved in Appendix  \ref{EQ_CONT_ASCOLI}.\hfill $\square $

 \section{Extension to higher order derivatives}
 \label{EXT}
 We explain here how the estimates \eqref{Derivatives_bounds}, \eqref{Second_bounds}, 
\eqref{Derivatives_bounds2} can be extended for higher order derivatives in our analysis. 
We claim that under {\bf (S)} the \textit{a-priori} bounds of Theorem \ref{CTR_DER_HK_SMOOTH_COEFF} can be obtained for any $j\in\N$, using the same techniques based on the Duhamel representation of the density and \eqref{GRA}. On the other hand the circular arguments used in Section \ref{SEC_MT_CIRC} can be repeated as well, provided that the coefficients are smooth enough. 

For instance, let us assume {\bf (S)} to be in force; assume as well that $\|\nabla\sigma\|_\infty+\|\nabla b\|_\infty<\infty$ and 
for some $\a,\beta\in (0,1]$, $\kappa_2\ge 1$, 
\begin{equation}\label{COND_FOR_THIRD}
|\nabla\sigma(t,x)-\nabla\sigma(t,y)|\le \kappa_2|x-y|^{\a}, \quad  |\nabla b(t,x)-\nabla b(t,y)|\le \kappa_2|x-y|^\beta, \quad x,y\in\R^d.
\end{equation}
We aim here at proving that we can obtain bounds on the third order derivatives which only depend on {\bf (H$^\sigma_1$)}, {\bf (H$^b_1$)} and the constants in \eqref{COND_FOR_THIRD}. Namely, we want to illustrate a kind of \textit{parabolic bootstrap property}, i.e. in  \eqref{COND_FOR_THIRD} we give some H\"older conditions on the first derivatives of the coefficients which together with the assumptions {\bf (H$^\sigma_1$)}, 
{\bf (H$^b_1$)} lead to a uniform control of the third order derivatives.

As in \eqref{DECOMP_DER_2_1}, for the choice of the freezing parameters $(\t,\x)=(s,x)$
and recalling $\tilde{p}_0(s,x,t,y)=\tilde{p}^{(s,x)}(s,x,t,y)$, we have the following representation for the derivatives of order three:  
\begin{equation}\label{DER3}
\nabla^3_x p(s,x,t,y)=-\nabla^3_y\tilde{p}_0(s,x,t,y)-\int^t_s\int_{\R^d}\nabla^3_z\tilde{p}_0(s,x,r,z)\cA^{(s,x)}_{r,z}p(r,z,t,y)\dif z\dif r.
\end{equation}
Let us now concentrate on the most singular term in \eqref{DER3}. Setting $u=(t+s)/2$ and $A^{(\t,\x)}_{r,z}$, $A^{\eps,(\t,\x)}_{r,z}$, 
$\bar A^{\eps,(\t,\x)}_{r,z}$ as in \eqref{ABBR_1}, \eqref{ABBR_2_BIS} we write
\begin{align}
&\int^t_s\int_{\R^d}\nabla^3_z\tilde{p}_0(s,x,r,z)\tr\left(A^{(s,x)}_{r,z}\cdot\nabla_z^2p(r,z,t,y)\right)\dif z\dif r\\
 =&\int_u^t \int_{\R^d}\nabla^3_z\tilde{p}_0(s,x,r,z)\tr\left((A_{r,z}^{\eps,(s,x)}
+\bar A_{r,z}^{\eps,(s,x)})\cdot\nabla_z^2p(r,z,t,y)\right)\dif z\dif r\\
&+   \int^u_s\int_{\R^d}\nabla^3_z\tilde{p}_0(s,x,r,z)\tr\left(A^{(s,x)}_{r,z}\cdot\nabla_z^2p(r,z,t,y)\right)\dif z\dif r 
=:G_1(s,x,t,y)+G_2(s,x,t,y).
\end{align}
When $r\in [u,t]$, $(r-s)^{-\frac 32}\asymp (t-s)^{-\frac 32}$ is not singular. Therefore we may control $G_1$ 
similarly to the terms $I_3$ and $I_4$ appearing in Section \ref{SEC_FOR_BOUNDING_SECOND_ORDER_DERIVATIVES}, owing to the fact that the upper bound on $\nabla_z^2p$ is already available at this point.

When $r\in [s,u]$, then $(r-s)^{-\frac 32}$ is indeed singular. Thus, to control $G_2$ 
the point is precisely to exploit the regularity of the coefficients and perform an integration by parts to balance the singularity. We write
\begin{align}
G_2(s,x,t,y)
&=-\int^u_s\int_{\R^d}\nabla^2_z\tilde{p}_0(s,x,r,z)\tr\left(\nabla_zA^{(s,x)}_{r,z}\cdot \nabla_z^2p(r,z,t,y)\right)\dif z\dif r\\
&\quad -\int^u_s\int_{\R^d}\nabla^2_z\tilde{p}_0(s,x,r,z)\tr\left(A^{(s,x)}_{r,z}\cdot\nabla_z^3p(r,z,t,y)\right)\dif z\dif r,
\end{align} 
and define 
$$f_3(s):=(t-s)^{\frac 32}\sup_{x,y}|\nabla^3_xp(s,x,t,y)|/\bar p(s,x,t,y);$$
Then, exploiting the uniform bounds for the derivatives of order lower or equal than $ 2$ obtained in Section \ref{SEC_MT_CIRC}, we eventually derive 
\begin{align}
f_3(t)\lesssim 1+\int_s^u(r-s)^{-1+\frac{\a}{2}}f_3(s)\dif s \Rightarrow \sup_{s\in [0,t]}f_3(s)\lesssim 1, 
\end{align}
which yields the desired estimate for $\nabla_x^3p$. In the same manner, starting from the Duhamel expansion \eqref{D1}, 
and assuming in addition that $\|\nabla^2\s\|_\infty<\infty$ and
$|\nabla^2\s(t,x)-\nabla^2\s(t,y)|\le \kappa_3|x-y|^\alpha$ for some $\alpha\in(0,1)$ we could derive 
$$
|\nabla^2_yp(s,x,t,y)|\lesssim (t-s)^{-1}\bar{p}(s,x,t,y).
$$

A careful reading of the proof suggests that the above arguments may be repeated for any derivative of order $j>3$ in the backward variable $x$ as soon as we have appropriate regularity assumptions on $\nabla^{j-2}\s$ and $\nabla^{j-2}b$. More precisely, 
assuming that 
$$
\|\nabla^{j'}\sigma\|_\infty+\|\nabla^{j'}b\|_\infty<\infty, \ j'=1,\cdots,j-2,
$$
and for some $\a,\beta\in (0,1]$, $\kappa_{j-2}\ge 1$,
$$
|\nabla^{j-2}\sigma(t,x)-\nabla^{j-2}\sigma(t,y)|\le \kappa_{j-2}|x-y|^{\a}, 
\quad  |\nabla^{j-2} b(t,x)-\nabla^{j-2} b(t,y)|\le \kappa_{j-2}|x-y|^{\beta}, \quad x,y\in\R^d,$$ 
then we may derive 
$$|\nabla^j_xp(s,x,t,y)|\lesssim (t-s)^{-\frac{j}{2}}\bar{p}(s,x,t,y).$$
On the other hand, the derivative with respect to the forward variable $\nabla_y^{j-1}$ requires an additional assumption on $\nabla^{j-1}\s$. Again, assuming that for some $\alpha\in(0,1)$,
$|\nabla^{j-1} \s(t,x)-\nabla^{j-1} \s(t,y)|\le \textcolor{red}{\kappa_{j-1}}|x-y|^\alpha$ for any $x,y\in\R^d$, then we may derive 
$$
|\nabla^{j-1}_yp(s,x,t,y)|\lesssim (t-s)^{-\frac{j-1}{2}}\bar{p}(s,x,t,y).
$$

 \appendix

\section{Proof of the equicontinuity \eqref{PREAL_ASCOLI1}, \eqref{PREAL_ASCOLI2} and \eqref{PREAL_ASCOLI3}}\label{EQ_CONT_ASCOLI}
Importantly, we mention that we  drop in this section the subscripts and superscripts in $\eps $ for notational convenience. However, it must be recalled that we aim at proving some equicontinuity properties for the densities associated with the SDE \eqref{SDE_1} with mollified coefficients and their derivatives.
 
In this section we devote to proving the following H\"older continuity of the derivatives.
\bl
Suppose that {\bf (H$^\sigma_{\alpha}$)} and {\bf (H$^b_\beta$)} hold.
Let $T>0$ and $\gamma_1\in(0,1)$ and $\gamma_2\in(0,\alpha)$ and $\gamma_3\in(0,\alpha\wedge\beta)$.
\begin{enumerate}
\item[{\bf (C$_1$)}] There exist constants $C,\lambda>0$ depending only on $\Theta$, $\gamma_1,\gamma_2$
such that for all $(s,t)\in\mD^T_0$ and $x,x',y,y'\in\mR^d$,
\begin{align}
\label{THE_FINAL_BD_DER_2_1}
|\nabla_xp(s,x,t,y)-\nabla_xp(s,x',t,y)|&\lesssim_C \frac{|x-x'|^{\gamma_1}}{(t-s)^{(1+\gamma_1)/2}}
\Big({\bf g}_\lambda(s,x,t,y)+{\bf g}_\lambda(s,x',t,y)\Big)\\
|\nabla_xp(s,x,t,y)-\nabla_xp(s,x,t,y')|&\lesssim_C \frac{|y-y'|^{\gamma_2}}{(t-s)^{(1+\gamma_2)/2}}\Big({\bf g}_\lambda(s,x,t,y)+{\bf g}_\lambda(s,x,t,y')\Big).
\end{align}
\item[{\bf (C$_2$)}] If $\beta\in (0,1] $, there exist constants $C,\lambda>0$ depending only on $\Theta$
such that for all $(s,t)\in\mD^T_0$ and $x,x',y\in\mR^d$,
\begin{align}
\label{THE_FINAL_BD_DER_2}
|\nabla_x^2p(s,x,t,y)-\nabla_x^2p(s,x',t,y)|
&\lesssim_C \left(\frac{|x-x'|}{(t-s)^{\frac32}}+\frac{|x-x'|^{\a}+|x-x'|^\beta}{t-s}\right)
\\&\qquad\times\Big({\bf g}_\lambda(s,x,t,y)+{\bf g}_\lambda(s,x',t,y)\Big),\\
|\nabla_x^2p(s,x,t,y)-\nabla_x^2p(s,x,t,y')|
&\lesssim_C \left(\frac{|y-y'|^{\gamma_2}}{(t-s)^{1+\frac{\gamma_2}2}}+\frac{|y-y'|^{\a}+|y-y'|^\beta}{t-s}\right)
\\&\qquad\times\Big({\bf g}_\lambda(s,x,t,y)+{\bf g}_\lambda(s,x',t,y)\Big).
\end{align}
\item[{\bf (C$_3$)}] If $\sigma$ also satisfies \eqref{BD_DINI_GRAD_SIGMA} and $\beta\in(0,1)$, then there exist constants $C,\lambda>0$ depending only on $\Theta$, $\gamma_1,\gamma_3$
such that for all $(s,t)\in\mD^T_0$ and $x,x',y,y'\in\mR^d$,
\begin{align}
\label{THE_FINAL_BD_DER_2_3}
|\nabla_yp(s,x,t,y)-\nabla_yp(s,x,t,y')|
\lesssim_C \frac{|y-y'|^{\gamma_3}}{(t-s)^{(1+\gamma_3)/2}}\Big({\bf g}_\lambda(s,x,t,y)+{\bf g}_\lambda(s,x,t,y')\Big),\\
|\nabla_yp(s,x,t,y)-\nabla_yp(s,x',t,y)|
\lesssim_C \frac{|x-x'|^{\gamma_1}}{(t-s)^{(1+\gamma_1)/2}}\Big({\bf g}_\lambda(s,x,t,y)+{\bf g}_\lambda(s,x,t,y')\Big).
\end{align}
\end{enumerate}
\el
\begin{proof}
We only prove {\bf (C$_2$)} and focus on the sensitivity w.r.t the variable $x$. The sensitivity w.r.t. the variable $y$ could be established similarly. Also, the inequalities in conditions {\bf (C$_1$)} and {\bf (C$_3$)}  could be shown more directly.
\\
\\
First of all, if $|x-x'|^2>(t-s)/4$, then by \eqref{Gr}, we clearly have
\begin{align}
|\nabla_x^2 p(s,x,t,y)-\nabla_x^2 p(s,x',t,y)|&\lesssim (t-s)^{-1}\Big({\bf g}_\lambda(s,x,t,y)+{\bf g}_\lambda(s,x',t,y)\Big)\lesssim \mbox{r.h.s. of \eqref{THE_FINAL_BD_DER_2}}.
\end{align}
Next we restrict to the so-called \textit{diagonal case} 
$$
|x-x'|^2\le (t-s)/4.
$$
For any fixed freezing point $(\tau,\xi)$ and $r\in (s,t)$, by \eqref{D0}, one sees that
\begin{align}\label{e0bis}
p(s,x,t,y)=\tilde{P}_{s,r}^{(\t,\xi)}p(r,\cdot,t,y)(x)+\int_s^r
\int_{\R^d}\tilde{p}^{(\t,\xi)}(s,x,u,z)\cA^{(\t,\xi)}_{u,z}p(u,z,t,y)\dif z \dif u,
\end{align}
where, with the notations of \eqref{SDE_FROZEN}, 
$$
\tilde{P}_{s,r}^{(\t,\xi)}f(x)=\int_{\R^d}\tilde p^{(\t,\xi)}(s,x,r,z)f(z)\dif z.
$$
Let us now differentiate w.r.t. $r$. We obtain for all $(\tau,\xi)\in [0,T]\times \R^d$:
\begin{align}
0 =\partial_r [\tilde{P}_{s,r}^{(\t,\xi)}p(r,\cdot,t,y)(x)]+
\int_{\R^d}\tilde{p}^{(\t,\xi)}(s,x,r,z)\cA^{(\tau,\xi)}_{r,z}p(r,z,t,y)\dif z .\label{THE_EQ_PREVIOUS_CHANGE_POINT}
\end{align}
Fix $s_0\in(s,t)$. Now, integrating \eqref{THE_EQ_PREVIOUS_CHANGE_POINT}
between $s_0$ and $t$ and taking $\xi=\xi' $, we get
\begin{align}
0 =\tilde{p}^{(\t,\xi')}(s,x,t,y)-\tilde{P}_{s,s_0}^{(\t,\xi')}p(s_0,\cdot,t,y)(x)
+\int_{s_0}^t \dif r \int_{\R^d}\tilde{p}^{(\t,\xi')}(s,x,r,z)\cA^{(\t,\xi')}_{r,z}p(r,z,t,y)\dif z.
\label{INT_DIAGONALE}
\end{align}
Moreover, integrating \eqref{THE_EQ_PREVIOUS_CHANGE_POINT} between $s$ and $s_0$, we obtain
\begin{align}
0 =\tilde{P}_{s,s_0}^{(\t,\xi)}p(s_0,\cdot,t,y)(x)-p(s,x,t,y)+\int_s^{s_0} \dif r \int_{\R^d}\tilde{p}^{(\t,\xi)}(s,x,r,z)
\cA^{(\tau,\xi)}_{r,z}p(r,z,t,y)\dif z ; \label{INT_FUORI_DIAGONALE}
\end{align}
 Summing up the two equalities we get the following new representation for $p(s,x,t,y)$:
 \begin{align}
p(s,x,t,y)=&\tilde{p}^{(\t,\xi')}(s,x,t,y)+\Big(\tilde{P}_{s,s_0}^{(\t,\xi)}- \tilde{P}_{s,s_0}^{(\t,\xi')}\Big)p(s_0,\cdot,t,y)(x)\\
&+\int_{s_0}^t \dif r \int_{\R^d}\tilde{p}^{(\t,\xi')}(s,x,r,z)\cA^{(\tau,\xi')}_{u,z}p(r,z,t,y)\dif z\\
&+\int_s^{s_0} \dif r \int_{\R^d}\tilde{p}^{(\t,\xi)}(s,x,r,z)
\cA^{(\tau,\xi)}_{r,z}p_{\eps}(r,z,t,y)\dif z,\label{THE_DENS_WITH_CHANGING_FREEZING_POINT}
\end{align} 
which, together with \eqref{D0} by taking $x=x'$ and $\xi=\xi'$ there, yields
\begin{align}
p(s,x,t,y)-p(s,x',t,y)=& \tilde{p}^{(\t,\xi')}(s,x,t,y)-\tilde{p}^{(\t,\xi')}(s,x',t,y)+\Big(\tilde{P}_{s,s_0}^{(\t,\xi)}- \tilde{P}_{s,s_0}^{(\t,\xi')}\Big)p(s_0,\cdot,t,y)(x)\\
&+\Delta_{{\rm diag}}^{\t,\xi',\xi'}(s,t,x,x',y) + \Delta_{{\rm off-diag}}^{\t,\xi,\xi'}(s,t,x,x',y),\label{STIMA_CON_PUNTI_DIVERSI}
\end{align} 
where
\begin{align}
\Delta_{{\rm diag}}^{\t,\xi',\xi'}(s,t,x,x',y)&=\int_{s_0}^{t}\dif r\int_{\R^d}\Big[\tilde{p}^{(\t,\xi')}(s,x,r,z)-\tilde{p}^{(\t,\xi')}(s,x',r,z)\Big]
\cA^{(\tau,\xi')}_{r,z}p(r,z,t,y)\dif z
\end{align}
and
\begin{align}
\Delta_{{\rm off-diag}}^{\t,\xi,\xi'}(s,t,x,x',y)&=\int^{s_0}_s \dif r\int_{\R^d}\Big[\tilde{p}^{(\t,\xi)}(s,x,r,z)
\mathcal{A}^{(\t,\xi)}_{r,z}p(r,z,t,y)-\tilde{p}^{(\t,\xi')}(s,x',r,z)
\cA^{(\tau,\xi')}_{r,z}p(r,z,t,y)\Big]\dif z.
\end{align}
Observe that for any freezing couple  $ (\tau,\xi)$ and $h\in\mR^d$,
\begin{align}
\nabla_x^2 \tilde p^{(\t,\xi)}(s,x+h,t,y)=\nabla_y^2 \tilde p^{(\t,\xi)}(s,x,t,y-h).
\end{align}
After differentiating twice in $x$ for both sides of \eqref{STIMA_CON_PUNTI_DIVERSI} and taking $\tau=s$ and $\xi=x,\ \xi'=x'$,
we obtain 
\begin{align}
\nabla_x^2p(s,x,t,y)-\nabla_x^2p(s,x',t,y)&= \sum_{i=1}^4I_i(s,t,x,x',y),\label{STIMA_CON_PUNTI_DIVERSI_DER}
\end{align} 
where, with the notation $\tilde{p}_0(s,x,t,y)=\tilde{p}^{(s,x)}(s,x,t,y)$,
\begin{align*}
I_1(s,t,x,x',y)&:=\nabla_y^2\tilde{p}_0(s,x',t,y+x'-x)-\nabla_y^2\tilde{p}_0(s,x',t,y),\\
I_2(s,t,x,x',y)&:=\int_{\R^d}\nabla^2_z(\tilde p^{(s,x)}(s,x,s_0,z)-\tilde p^{(s,x')}(s,x,s_0,z))p(s_0,z,t,y)\dif z,\\
I_3(s,t,x,x',y)&:=\int_{s_0}^{t}\dif r\int_{\R^d}\Big[\nabla^2_z\tilde{p}_0(s,x',r,z+x'-x)-\nabla^2_z\tilde{p}_0(s,x',r,z)\Big]\cA^{(s,x')}_{r,z}p(r,z,t,y)\dif z,\\
I_4(s,t,x,x',y)&:=\int^{s_0}_s \dif r\int_{\R^d}\Big[\nabla^2_z\tilde{p}_0(s,x,r,z)
\mathcal{A}^{(s,x)}_{r,z}p(r,z,t,y)-\nabla^2_z\tilde{p}_0(s,x',r,z)\cA^{(s,x')}_{r,z}p(r,z,t,y)\Big]\dif z.
\end{align*}
Note that by Lemma \ref{Le21}, for $j\in\mN$ and $h\in\mR^d$ with $|h|^2\leq (t-s)/4$,
\begin{align}
\Big|\nabla_y^j\tilde{p}_0(s,x,t,y+h)-\nabla_y^j\tilde{p}_0(s,x,t,y)\Big|
&\le |h|\sup_{\r\in [0,1]}\big|\nabla^{j+1}_y\tilde{p}_0(s,x,t,y+\r h)\big| \\
&\lesssim |h|(t-s)^{-(j+1)/2}\sup_{\r\in [0,1]} g_\lambda(t-s,\theta_{t,s}(x)+\r h-y)\\
&\lesssim |h|(t-s)^{-(j+1)/2}g_\lambda(t-s,\theta^{(1)}_{t,s}(x)-y),\label{eC2_1}
\end{align}
using Lemma \ref{Pr1} for the last step. On the other hand, we also have
$$
\Big|\nabla_y^j\tilde{p}_0(s,x,t,y+h)\Big|\lesssim (t-s)^{-j/2}g_\lambda(t-s,\theta^{(1)}_{t,s}(x)-y).
$$
Thus, by interpolation, we get for any $\gamma\in(0,1)$,
\begin{align}
\Big|\nabla_y^j\tilde{p}_0(s,x,t,y+h)-\nabla_y^j\tilde{p}_0(s,x,t,y)\Big|\lesssim |h|^\gamma(t-s)^{-(j+\gamma)/2}g_\lambda(t-s,\theta^{(1)}_{t,s}(x)-y).
\end{align}
Hence,
\begin{align}
|I_1(s,t,x,x',y)|
\lesssim |x-x'|^\gamma(t-s)^{-1-\frac{\gamma}2}{\bf g}_\lambda(s,x',t,y).
\end{align}
To treat the remaining terms, we take $s_0=s+|x-x'|^2$. We have the following claim:
\begin{equation}
\left|\tilde{p}^{(s,x)}(s,x,s_0,y)-\tilde{p}^{(s,x')}(s,x,s_0,y)\right|\lesssim\big(|x-x'|^{\a}+|x-x'|^\beta\big){\bf g}_\lambda (s,x,s_0,y).\label{SENSI_DIFF_DENS_FREEZING}
\end{equation}
Indeed, by Lemma \ref{Pr1}, there is a constant $C=C(\Theta)$ such that
\begin{equation}\label{e_apx0}
|\theta_{r,s}(x)-\theta_{r,s}(x')|\lesssim_C |x-x'|+|r-s|, \quad x,x'\in\R^d, \, r\in[s,s_0].
\end{equation}
Recalling $\vartheta^{(\t,\xi)}_{s_0,s}=\int^{s_0}_s b(r,\theta_{r,\tau}(\xi))\dif r$ from the notations of Section \ref{SEC_FROZEN_PLUS_DUHAMEL},  we have:
\begin{align}\label{e_apx0bis}
|\vartheta^{(s,x)}_{s_0,s}-\vartheta^{(s,x')}_{s_0,s}|&
\le \int_s^{s_0}|b(r,\theta_{r,t}(x))-b(r,\theta_{r,t}(x'))|\dif r \\
& \lesssim\int_s^{s_0}|\theta_{r,s}(x)-\theta_{r,s}(x')|^\beta\dif r\lesssim |x-x'|^{2+\beta},
\end{align}
where the last step is due to $|r-s|\leq|x-x'|^2\leq|t-s|/4$.
The desired claim \eqref{SENSI_DIFF_DENS_FREEZING} follows by \eqref{GAU} and elementary but cubersome calculations.

Now,  integrating by parts, we get from \eqref{BD_HESS_PARAM}, \eqref{SENSI_DIFF_DENS_FREEZING} and  Lemma \ref{CKE}
\begin{align}
|I_2(s,t,x,x',y)|&\leq\int_{\R^d}|\tilde p^{(s,x)}(s,x,s_0,z)-\tilde p^{(s,x')}(s,x,s_0,z)|\cdot|\nabla^2_zp(s_0,z,t,y)|\dif z,\\
&\lesssim\big(|x-x'|^{\a}+|x-x'|^\beta\big)(t-s_0)^{-1}
\int_{\R^d}{\bf g}_\lambda (s,x,s_0,z){\bf g}_{\lambda'} (s_0,z,t,y)\dif z\\
&\lesssim\big(|x-x'|^{\a}+|x-x'|^\beta\big)(t-s)^{-1}{\bf g}_{\lambda''} (s,x,t,y).
\end{align}
For $I_3$, by \eqref{eC2_1} and using  arguments completely similar to those of Section \ref{SEC_FOR_BOUNDING_SECOND_ORDER_DERIVATIVES}, we have
\begin{align}
|I_3(s,t,x,x',y)|\lesssim |x-x'|^{\a}(t-s)^{-1}{\bf g}_{\lambda''} (s,x,t,y).
\end{align}
Finally, for $I_4$, from \eqref{AQ77}, we have
 \begin{align}
|I_4(s,t,x,x',y)|&\lesssim \int_{s}^{s_0}\int_{\R^d}
\frac{({\bf g}_{\lambda_1}(s,x,r,z)+{\bf g}_{\lambda_1}(s,x',r,z)){\bf g}_{\lambda_2}(r,z,t,y)}{(r-s)^{1-\frac \a2}(t-r)}\dif z\dif r\\
&\lesssim \big({\bf g}_{\lambda_3}(s,x,t,y)+{\bf g}_{\lambda_3}(s,x',t,y)\big)\int_{s}^{s_0}\frac{\dif r}{(r-s)^{1-\frac \a2}(t-r)}\\
&\lesssim \frac{|x-x'|^{\a}}{t-s}\big({\bf g}_{\lambda_3}(s,x,t,y)+{\bf g}_{\lambda_3}(s,x',t,y)\big),
\end{align}
where we have used that $|x-x'|^2\leq (t-s)/4$ and $s_0=s+|x-x'|$. Combining the above calculations, we obtain \eqref{THE_FINAL_BD_DER_2}.
\end{proof}

 \section*{Acknowledgment}
{Research of S. Menozzi is funded by the Russian Academic Excellence Project '5-100'. 
Research of X. Zhang is funded by NNSFC grant of China (No. 11731009)  and the DFG through the CRC 1283 
``Taming uncertainty and profiting from randomness and low regularity in analysis, stochastics and their applications''. }
 
\bibliographystyle{acm}
\bibliography{bib_aggiornata}
\end{document}